\documentclass[11pt, reqno]{amsproc} 

\usepackage{geometry}
\usepackage{color}
\usepackage{verbatim}
\usepackage[utf8x]{inputenc}
\usepackage{t1enc}
\usepackage[english]{babel}

\usepackage{amsmath,amssymb,amsthm}
\usepackage{mathrsfs,esint}
\usepackage{graphicx,wrapfig}
\usepackage[format=hang]{caption}
\usepackage{hyperref}
\usepackage{orcidlink}

\newcommand*{\R}{{\mathbb R}}

\newcommand*{\N}{{\mathbb N}}

\newcommand*{\eps}{\varepsilon}

\newcommand*{\pip}{\varphi}
\newcommand*{\phii}{\varphi}
\newcommand*{\wtil}{\widetilde}
\newcommand*{\phiinbar}{\overline{\phii^{-1}}}

\providecommand*{\vint}[1]{\mathchoice
          {\mathop{\vrule width 5pt height 3 pt depth -2.5pt
                  \kern -9pt \kern 1pt\intop}\nolimits_{\kern -5pt{#1}}}
          {\mathop{\vrule width 5pt height 3 pt depth -2.6pt
                  \kern -6pt \intop}\nolimits_{\kern -3pt{#1}}}
          {\mathop{\vrule width 5pt height 3 pt depth -2.6pt
                  \kern -6pt \intop}\nolimits_{\kern -3pt{#1}}}
          {\mathop{\vrule width 5pt height 3 pt depth -2.6pt
                  \kern -6pt \intop}\nolimits_{\kern -3pt{#1}}}}

\DeclareMathOperator{\dist}{dist}
\DeclareMathOperator{\diam}{diam}

\DeclareMathOperator{\pa}{\sqcup}

\numberwithin{equation}{section}
\theoremstyle{plain}
\newtheorem{thm}[equation]{Theorem}

\newtheorem{prop}[equation]{Proposition}
\newtheorem{cor}[equation]{Corollary}
\newtheorem{lem}[equation]{Lemma}

\theoremstyle{definition}

\newtheorem{defn}[equation]{Definition}

\newtheorem{remark}[equation]{Remark}

\newtheorem{example}[equation]{Example}

\newtheorem{assumptions}[equation]{Assumptions}
\begin{document}

\title[Warped products: hyperbolicity, visual boundary]
{Warped product spaces: Gromov hyperbolicity and identification of the visual boundary} 
\author{Josh Kline\, \orcidlink{0000-0002-5916-5438}}
\address{Department of Mathematical Sciences, P.O.~Box 210025, University of Cincinnati, Cincinnati, OH~45221-0025, U.S.A.}
\email{klinejp@ucmail.uc.edu}
\author{Nageswari Shanmugalingam\, \orcidlink{0000-0002-2891-5064}}
\address{Department of Mathematical Sciences, P.O.~Box 210025, University of Cincinnati, Cincinnati, OH~45221-0025, U.S.A.}
\email{shanmun@uc.edu}
\author{Gareth Speight\, \orcidlink{0000-0002-3028-1558}}
\address{Department of Mathematical Sciences, P.O.~Box 210025, University of Cincinnati, Cincinnati, OH~45221-0025, U.S.A.}
\email{gareth.speight@uc.edu}
 \author{Yi Wang\, \orcidlink{0009-0009-9469-3402}}
\address{Department of Mathematical Sciences, P.O.~Box 210025, University of Cincinnati, Cincinnati, OH~45221-0025, U.S.A.}
\email{wang8y5@mail.uc.edu}
\maketitle

\begin{abstract}
In this paper, we consider warped product spaces $X\times_{\phii}Y$, where $X$ is a complete geodesic Gromov hyperbolic space, $Y$ 
is a compact geodesic metric space, and the warping function $\phii$ satisfies suitable exponential growth conditions. We prove that the 
warped product is Gromov hyperbolic and derive an explicit estimate for its hyperbolicity constant. We further establish a homeomorphism 
between its Gromov boundary and $\partial_GX\times Y$,  
with an explicit comparison formula for the visual metric on the Gromov boundary of $X\times_\pip Y$ in 
terms of the visual metric $d_{\eps, X}$ on $\partial_GX$ and $d_Y$ for  suitable $\eps>0$.
\end{abstract}


\medskip

\noindent
    {\small Mathematics Subject Classification (2020):
Primary: 53C23.
Secondary: 30L05, 30L10, 51F30. 
}

\tableofcontents

\section{Introduction}

The most natural 
way of combining two Riemannian manifolds is through the Cartesian product of their respective metrics; if $M_1$, 
$M_2$ are
Riemannian manifolds and $g_1,g_2$ are their respective metrics, then the Cartesian product metric on $M_1\times M_2$ is given by 
\[
g(\vec{V},\vec{W})=g_1(\pi_1(\vec{V}),\pi_1(\vec{W}))+g_2(\pi_2(\vec{V}),\pi_2(\vec{W})),
\] 
where $\pi_1$ and $\pi_2$ are the projection maps from
$T_{(x_1,x_2)}M_1\times M_2$ to $T_{x_1}M_1$ and $T_{x_2}M_2$ respectively. More geometric variants are introduced by
considering warped products, where a warping function $f$ on $M_1$ is employed to obtain the metric 
\[
g(\vec{V},\vec{W})=g_1(\pi_1(\vec{V}),\pi_1(\vec{W}))+f(x_1)^2\,g_2(\pi_2(\vec{V}), \pi_2(\vec{W})).
\] 
See for instance~\cite{BO, Brendle, Chen, Ogawa, Tsukada} for a small sampling of the extensive literature on warped product spaces. 
With the warping function
$f$ chosen to be convex, the resulting warped product space exhibits negative curvature, as in~\cite[Theorem~7.5]{BO}.

The notion of warped product spaces makes sense for metric spaces as well, for we can define the warped product distance
by minimizing the warped length of curves connecting pairs of points in the product space, with the warping of length in one of the 
component spaces.
In the setting of 
Alexandrov spaces (of curvature bounded below
or curvature bounded above), a splitting theorem was considered by
Alexander and Bishop in~\cite[Theorem~2.1, Theorem~2.2]{AB} with the split a warped product, see also~\cite{AB1, Chen}.
Here, one of the component spaces used in the warped product is an unbounded interval. More general warped product spaces
were later considered by Alexander and Bishop in~\cite{AB2}, where the component metric spaces are considered to be 
more general Alexandrov spaces. In the context of synthetic Ricci curvature conditions on metric measure spaces, 
warped splitting theorems for ${\rm RCD}(-(N-1),N)$ spaces were established in~\cite{GigMarc}.

The Alexandrov curvature condition is quite strong as it applies at all scales in the space. In the present note, we investigate
warped product spaces from the point of view of a more large-scale negative curvature condition, namely 
Gromov hyperbolicity. 
This investigation was motivated by the work in~\cite{KKSS}, 
which investigated how warped spaces affect the behavior of homogeneous Sobolev-type spaces. One of the warping spaces
considered in~\cite{KKSS} was a half-line, so the verification of Gromov hyperbolic geometry in that setting was somewhat 
straightforward. In the present paper we consider more general geodesic Gromov hyperbolic spaces as one of the warping 
spaces. In this generality, verifying that the warped spaces are Gromov hyperbolic is made more complicated by the fact that 
geodesics in general Gromov hyperbolic spaces need not be unique. This paper is devoted to the 
verification of Gromov hyperbolicity and identification of the visual boundary of the warped product space. This work lays
the foundation for the study of potential theoretic aspects of warped product spaces
in the
spirit of~\cite{BGM, D, GGN, KKSS}, a line of investigation the 
authors are currently pursuing. Note that the results of~\cite{GGN} are more general as they consider a wider class of warping functions
than we do, for they do not consider the Gromov hyperbolic structure.

We now state our main results. See Section~\ref{background}
for the definition of the warped space $X \times_{\pip}Y$ and other relevant background. 
Here, the warping function $\pip:[0,\infty)\to[1,\infty)$ is assumed to satisfy the conditions given in Assumptions~\ref{ass}, and in particular, is
assumed to grow at least as fast as an exponential function.
The following is our first main theorem.

\begin{thm}\label{thm:Main1}
Suppose that $(X,d_X)$ is a complete geodesic Gromov $\delta$-hyperbolic metric space, and $(Y,d_Y)$ is a compact 
geodesic metric space.
Let $\pip \colon [0,\infty)\to[1,\infty)$ satisfy the hypotheses of 
Assumptions~\ref{ass} for some $\alpha>0$. Then the warped space $X\times_\pip Y$
is a Gromov $2\delta_\pip$--hyperbolic geodesic space with 
\[
\delta_\pip=\tfrac{1}{\alpha}+44\delta+\tfrac32\pip(0)\, \diam(Y).
\]
\end{thm}

To prove Theorem~\ref{thm:Main1}, in Section~\ref{Sec:ProofMain1} we will prove~\eqref{gromovdelta} 
for the choice of base-point $z_0=(x_0,y_0)$, with $\delta_{\pip}$ playing the role of $\delta$ there; see~\eqref{product-compare}.
The claim of Theorem~\ref{thm:Main1} follows from~\cite[Lemma~2.1.5]{BuSch}, see also the discussion in the next section.

It is notable that we only need the space on which the warping function is defined to be geodesic and Gromov hyperbolic, 
while the other component metric space is merely a compact geodesic space. This is in contrast to~\cite[Theorem~2.3]{AB2} 
or~\cite[Theorem~1]{Chen}, where both component spaces are needed to be Alexandrov in order for the warped product space to be Alexandrov.
 With the choice of $Y$ as the Euclidean $n$-dimensional
unit sphere in one of the  constructions considered in~\cite{KKSS}, the warped product space is \emph{not} an Alexandrov space
with negative curvature upper bound, as the unit sphere, which is not of negative curvature, has an isometric embedding into the
warped space.

Theorem~\ref{thm:Main1} is an extension of~\cite[Theorem~1.7]{KKSS} which 
explored the case where $X$ is the Euclidean ray $[0,\infty)$, a naturally
Gromov hyperbolic space. The proof of Gromov hyperbolicity in this context, as given in~\cite{KKSS}, benefits from the fact that
given any three points in the ray there is exactly one geodesic connecting any two of those points and that geodesic extends to
a geodesic that also passes through the third point. This convenient phenomenon is not available to us in our current setting. This 
leads to a more complicated description of geodesics in the warped space, see Proposition~\ref{lem:shape-phi-geods-Old}, and 
consequently complicates the proof of Theorem~\ref{thm:Main1} compared to that of~\cite{KKSS}.

Having established the Gromov hyperbolicity 
of the warped product space $X\times_\pip Y$, we next identify its visual boundary, also known as
the Gromov boundary.  The metrics $d_{\eps, Z}$ and $d_{\eps,X}$ are
visual metrics on the respective Gromov boundaries corresponding to the parameter $\eps>0$, see~\cite[page~14]{BuSch}
or Subsection~\ref{Sub:GromovBDRY} in the next
section for their definition. Given two quantities $F$ and $H$, we say that $F\simeq H$ if there is a constant $C\ge 1$ such that
$C^{-1}\, F\le H\le C\, F$, and the constant $C$ is called a comparison constant.
The following is our second main theorem, establishing the Gromov boundary of the warped product space $Z=X\times_\pip Y$
as the Cartesian product of the Gromov boundary $\partial_GX$ of $X$ with $Y$.

\begin{thm}\label{thm:Main2}
For $Z=X\times_\phii Y$, we have the following:
\begin{enumerate}
    \item There exists a bijection $\Phi =(\Phi_1,\Phi_2) \colon \partial_GZ\to\partial_GX\times Y$. 
    \item For all $0<\eps\le \log_e(2)/\delta_\pip$, the map $\Phi$ is a homeomorphism between $(\partial_GZ,\,d_{\eps,Z})$ and 
    $(\partial_GX\times Y,\,d_{\eps,X}\times d_Y)$, and for all $z,\wtil z\in\partial_GZ$ we have
    \[
    d_{\eps,Z}(z,\wtil z)\simeq d_{\eps,X}(\Phi_1(z),\Phi_1(\wtil z))
      +\exp\left\{-\eps\phiinbar\left(\frac{2}{\alpha d_Y\left(\Phi_2(z),\Phi_2(\wtil z)\right)}\right)\right\}.
    \]
   Here $\phiinbar$ is as in Remark~\ref{pipremark}, and
   the comparison constant depends only on $\alpha$, $\delta$, $\diam(Y)$, $\phii(0)$, and the quantity $\phiinbar(\phii'(0)/\alpha)$.
\end{enumerate}
\end{thm}

The proof of Theorem~\ref{thm:Main2} can be found in Section~\ref{Sec:ProofMain2}. Theorem~\ref{thm:Main2} also 
tells us that in taking
warped-products of $X$ with $Y$ we obtain a richer class of geometric structure at large scale in comparison to 
mere Cartesian product; see
the discussion in Example~\ref{example:Cartesian}. We also point out that the restriction of $\eps$ in this theorem may not always be
optimal - this restriction is only to guarantee that there is a visual metric $d_{\eps,Z}$ on $\partial_GZ$. 
There may be situations where a larger
range of $\eps$ would be possible, see for instance Example~\ref{example:exponential}. Note that a change in the warping function $\pip$
can result in a change in the quasisymmetric class of the visual boundary $\partial_GZ$. This is addressed in the discussions found in
Example~\ref{example:exponential} with Example~\ref{example:superexponential}.

\vskip .3cm

\noindent{\bf Funding acknowledgement:} N.S.'s research is partially supported by the National Science Foundation (U.S.A.) grant DMS~\#2348748. 
G.S.'s research is partially supported by the National Science Foundation (U.S.A.)
 grant DMS~\#2348715. Y.W.'s research is partially supported by a Taft Summer Fellowship grant from the Taft Foundation.

\section{Background}\label{background}

In this section we provide a basic description of the two concepts studied in this paper--Gromov hyperbolicity and warped products.

\subsection{Metric space terminology}

Let $(Z,d_Z)$ be a metric space. For any $x\in Z$ and $r>0$, we denote the open ball with center $x$ and radius $r$ by 
$B(x,r):=\{w\in Z\, :\, d_Z(w,z)<r\}$. Similarly we denote the closed ball by $\overline{B}(x,r)=\{w\in Z\, :\, d_Z(w,z)\le r\}$. 
Given a set $A\subset Z$ and a point $z\in Z$, we denote the distance from $z$ to $A$ by
\[
\dist_Z(z, A):=\inf\, \bigg\lbrace d_Z(z,w)\, :\, w\in A\bigg\rbrace.
\]

A curve $\gamma$ in $Z$ is a continuous map $\gamma:[a,b]\to Z$ for some interval $[a,b]\subset\R$. 
We define the {\em length} of such a curve by
\[
\ell_{d_Z}(\gamma):=\sup\bigg\lbrace\sum_{i=1}^nd_Z(\gamma(t_i),\gamma(t_{i-1}))\, :\, n\in\N\text{ and }a=t_0<\cdots<t_n=b\bigg\rbrace.
\]
A curve is {\em rectifiable} if its length is finite. 
If $\gamma:[a,b]\to Z$ is a rectifiable curve, then for each $t\in[a,b]$ we can set
$s(t):=\ell_{d_Z}(\gamma\vert_{[a,t]})$. As $s$ is a monotone increasing function, it is differentiable almost everywhere in $[a,b]$,
and we set $|\gamma^\prime(t)|:=s^\prime(t)$ whenever it exists. We say that $\gamma$ is absolutely continuous if it is rectifiable
and the length function $s:[a,b]\to[0,\ell_{d_Z}(\gamma)]$ is absolutely continuous. When $\gamma:[a,b]\to Z$ is absolutely continuous and
$g:Z\to[0,\infty]$ is a Borel function, we set
\[
\int_\gamma g\, ds:=\int_a^b g(\gamma(t))\, |\gamma^\prime(t)|\, dt.
\]
The above definition is independent of the choice of parametrization of $\gamma$ as long as the parametrizations are absolutely
continuous, see~\cite[Section~5.1]{HKSTbook}. A comprehensive description of 
path integrals is given in~\cite[Chapter~1]{Vaisala1971}; this
book is set in the context of Euclidean spaces, but the discussion in the first chapter of~\cite{Vaisala1971} is 
applicable in the wider context of
metric spaces.
We refer to the image of a curve as its {\em trajectory}. 
If $\gamma$ is a curve and $p\in Z$ we may write $p\in \gamma$ to mean that $p$ is in the trajectory of $\gamma$.

If $\gamma:[a,b]\to Z$ and $\beta:[c,d]\to Z$ are two curves such that $\gamma(b)=\beta(c)$, then the {\em concatenation} 
$\gamma*\beta$ of these two curves is the map $\gamma*\beta:[a,b+(d-c)]\to Z$ with 
\[
\gamma*\beta(t)=\begin{cases}\gamma(t) &\text{ if }a\le t\le b,\\
           \beta(t-b+c) &\text{ if }b\le t\le b+(d-c).\end{cases}
\]

We say that $(Z,d_Z)$ is a {\em geodesic space} if for each $z_1,z_2\in Z$ there is a rectifiable curve $\gamma:[a,b]\to Z$ with 
$\gamma(a)=z_1$, $\gamma(b)=z_2$, and $\ell(\gamma)=d_Z(z_1,z_2)$. Such a curve is called a {\em geodesic} from $z_{1}$ to $z_{2}$. 
In geodesic spaces we will often use the notation $[z_{1},z_{2}]$ to denote a choice of geodesic between $z_{1}$ and $z_{2}$. 
We emphasize that this is a choice: geodesics need not necessarily be unique.
The notation $[z_1,z_2]$ may even refer to different choices of geodesics at different occurences in the paper.

\subsection{Definition of Gromov hyperbolicity and Gromov products.}

\begin{defn}\label{gromovproduct}
Let $(X,d_X)$ be a metric space. 
The \emph{Gromov product with respect to  point $x_0\in X$}  is the map from $X\times X$
to $[0,\infty)$ given by 
\begin{equation*} %\label{eq:Gromov-prod}
(x|y)_{x_0}:=\frac12\left[d_X(x,x_0)+d_X(y,x_0)-d_X(x,y)\right].
\end{equation*}
Given $\delta\ge 0$, 
we say that $(X,d_X)$ is \emph{Gromov $\delta$-hyperbolic} if for $x,y,z,w\in X$ we have
\begin{equation}\label{gromovdelta}
(x|y)_w+\delta\ge \min\{(x|z)_w,\ (y|z)_w\}.
\end{equation}
\end{defn}

Note that, by the triangle inequality, the Gromov product is always non-negative. 
We refer the reader to~\cite{CDP, VaisalaExpo} for more on the topic of Gromov hyperbolicity. 

\begin{remark}\label{baseswitch}
If \eqref{gromovdelta} holds for some $w\in X$ and all $x,y,z\in X$, then $(X,d_X)$ is 
Gromov $2\delta$-hyperbolic, see for instance~\cite[Proposition~1.2]{CDP}~or~\cite[Lemma~2.1.5]{BuSch}.
\end{remark}

When $(X,d_X)$ is a Gromov $\delta$-hyperbolic geodesic space, $x,y,z\in X$, and $[x,y]$, $[y,z]$, and $[z,x]$ are three geodesics,
then $[z,x]\subset\bigcup_{p\in[x,y]\cup[y,z]}B(p,2\delta)$, and so, 
geodesic triangles are $2\delta$-thin. Suppose $x, y, w\in X$, $[x,y]$ is a geodesic from 
$x$ to $y$, and $\overline{x}\in [x,y]$ such that $\dist_X(w,[x,y])=d_X(w,\overline{x})$. 
Then by~\cite[Estimate~2.33]{VaisalaExpo} it follows that
\begin{equation}\label{eq:gromov-prod2xbar}
d_X(w,\overline{x})-2\delta\le (x|y)_w\le d_X(w,\overline{x}).
\end{equation}

The standard model Gromov hyperbolic space is a metric tree, which is necessarily $0$-hyperbolic. The following theorem, 
from~\cite[Th\'eor\`eme~1, page~91]{CDP}, tells us that geodesic Gromov hyperbolic spaces have a skeletal tree structure (\cite{HSX}
has a statement of this theorem in English). Note that by union of geodesics we simply mean the union of their images.

\begin{thm}\label{eq:Treez}
Suppose that $(X,d)$ is a Gromov $\delta$-hyperbolic geodesic 
 space for some $\delta>0$, and let $x_0\in X$. Suppose that
$x_1,\cdots, x_n\in X\setminus\{x_0\}$ are $n$ distinct points. 
For each $i=1,\cdots, n$, let $\gamma_i$ be any geodesic in $X$
beginning at $x_0$ and terminating at $x_i$. Setting $W$ to be the union of these geodesics, there is a 
metric tree $T(W)$ and a continuous
map $u:W\to T(W)$ satisfying the following properties:
\begin{enumerate}
\item $T(W)$ is equipped with the natural length metric $h$.
\item For each $i=1,\cdots, n$, the restriction of $u$ to $\gamma_i$ is an isometry.
\item For each $x,y\in W$ we have that 
\begin{equation*} 
 d(x,y)-2k\delta\le h(u(x),u(y))\le d(x,y).
\end{equation*}
Here $k$ is any positive integer such that $2^k\ge 2n-1$.
\end{enumerate}
\end{thm}

\begin{remark}
Note that if $w_1,w_2\in X$ and $\gamma_1, \gamma_2$ are two geodesics in $X$ from $w_1$ to $w_2$, then by the Gromov 
$\delta$-hyperbolicity of
$X$ together with~\cite[Tripod Lemma~2.15, page~191]{VaisalaExpo}, necessarily $d_X(x_t,\widetilde{x}_t)\le 4\delta$ where 
$x_t\in\gamma_1$ and $\widetilde{x}_t\in \gamma_2$ with $d_X(x_t,w_1)=d_X(\widetilde{x}_t,w_1)=t$. 
\end{remark}

Using the above remark, we extend the above theorem to more degenerate cases where not all the points $x_0,x_1,\cdots, x_n$ are distinct.
For instance, if $x_i=x_j$ for some $i\ne j$, then with $\gamma_i$ and $\gamma_j$ as in the statement of the theorem, we can
construct the tree with the number of distinct points (which is at most $n+1$) and with $x_j$ and $\gamma_j$ deleted from the list, and
obtain the map $u$ from the theorem. We can then extend $u$ to $\gamma_j$ as well by setting $u(w):=u(x_w)$ 
for $w\in\gamma_j$, where
$x_w\in\gamma_i$ such that $d_X(x_0,w)=d_X(x_0,x_w)$. 

\begin{cor}\label{thm:tree}
Suppose that $(X,d)$ is a Gromov $\delta$-hyperbolic geodesic 
 space for some $\delta>0$. Suppose that
$x_0,x_1,\cdots, x_n\in X$ are $n+1$ points. 
For each $i=1,\cdots, n$, let $\gamma_i$ be any geodesic in $X$
beginning at $x_0$ and terminating at $x_i$. Setting $W$ to be the union of these geodesics, there is a metric tree $T(W)$ and a continuous
map $u:W\to T(W)$ satisfying the following properties:
\begin{enumerate}
\item $T(W)$ is equipped with the natural length metric $h$.
\item For each $i=1,\cdots, n$, the restriction of $u$ to $\gamma_i$ is an isometry.
\item For each $x,y\in W$ we have that 
\begin{equation*} 
 d(x,y)-2k\delta\le h(u(x),u(y))\le d(x,y).
\end{equation*}
Here $k$ is any positive integer such that $2^k\ge \max\{4,\, 2n-1\}$.
\end{enumerate}
\end{cor}

\subsection{Gromov boundary}\label{Sub:GromovBDRY}

\begin{defn}\label{gromovboundary}
Let $(X,d_X)$ be Gromov $\delta$-hyperbolic and $w_0\in X$ be a fixed base point. 
A sequence $(x_i)_{i\in\N}$ in $X$ is a \emph{Gromov sequence} if $\lim_{i,j\to\infty}(x_i|x_j)_{w_0}=\infty$. 

The \emph{Gromov boundary} $\partial_G X$ of $X$ consists of all Gromov sequences of $X$ under the 
equivalence relation that $(x_i)_{i\in\N} \sim(y_i)_{i\in\N}$ if and only if $\lim_{i\to\infty}(x_i|y_i)_{w_0}=\infty$.  
We extend the Gromov product to $\partial_GX$ by defining for all $\overline x,\overline y\in\partial_GX$:
\begin{equation}\label{eq:boundary gromov product}
    (\overline x|\overline y)_{w_0}:=\inf\left\{\liminf_{i,j\to\infty}(x_i|y_j)_{w_0}:(x_i)_{i\in\N} \in\overline x,\,(y_i)_{i\in\N}\in\overline y\right\}.
\end{equation}  
\end{defn}

An alternate construction of Gromov boundary can be found in~\cite[Part~III, Chapter~3]{BridsonHaefliger}.

For each $\eps>0$, we obtain a pseudometric $d_\eps$ on $\partial_GX$ by
\begin{equation}\label{eq:visual metric}
d_\eps(\overline x,\overline y)=\inf\left\{\sum_{i=1}^k\, e^{-\eps(\overline x_{i-1}|\overline x_i)_{w_0}}\, :\,
 k\in\N,\, \overline x_0,\overline x_1,\dots,\overline x_k\in\partial_GX,\,\overline x_0=\overline x,\,\overline x_k=\overline y\right\}
\end{equation}
for all $\overline x,\overline y\in\partial_GX$.  
By~\cite[Lemma~2.2.5, Proposition~2.2.6]{BuSch} ,
if $0<\eps\le \log_e(2)/\delta$, 
then $(\partial_GX,d_\eps)$ is a metric space and 
\begin{equation}\label{eq:visual metric comparison}
    \frac{1}{2\, e^{\delta \eps}}\, e^{-\eps(\overline x|\overline y)_{w_0}}
    \le d_\eps(\overline x,\overline y)\le e^{-\eps(\overline x|\overline y)_{w_0}}
\end{equation}
for all $\overline x,\overline y\in\partial_GX$.

\begin{defn}\label{def:VisualBdy}
A metric $\rho$ on $\partial_GX$, with $X$ a Gromov $\delta$-hyperbolic space, is said to be a \emph{visual metric} if there is a point $w_0\in X$
and there exist positive real numbers
$c,C$, and $\eps$ such that for $\xi,\xi^\prime\in\partial_GX$ we have
\[
c\, e^{-\eps(\xi|\xi^\prime)_{w_0}}\le \rho(\xi,\xi^\prime)\le C\, e^{-\eps(\xi|\xi^\prime)_{w_0}}.
\]
\end{defn}

The above discussion shows that when $\eps<\log_e(2)/\delta$, the metric $d_{\eps}$ as defined above is a visual metric.

\subsection{Warped product spaces}

Throughout the remainder of the paper we make the following assumptions.

\begin{assumptions}\label{ass}
Let $(X,d_{X})$ be a complete $\delta$-hyperbolic geodesic space with a fixed 
point $x_{0}\in X$ and let $(Y,d_{Y})$ be a compact geodesic space. 

We fix a function $\varphi:[0,\infty)\to[1,\infty)$ which satisfies the following:
\begin{enumerate}
\item[{\bf (a)}] $\varphi$ is of class $C^1$,
\item[{\bf (b)}] $\varphi^\prime\ge \alpha \varphi$ for some fixed $\alpha>0$,
\item[{\bf (c)}] $\varphi^\prime$ is strictly increasing. 
\end{enumerate}
We call such a function $\pip$ a \emph{warping function}.
\end{assumptions}

We can now define the warped product distance. In what follows, we consider the product topology on $X\times Y$.
We set $\Pi_X:X\times Y\to X$ and $\Pi_Y:X\times Y\to Y$ be the standard projection maps, that is, $\Pi_X((x,y))=x$ and
$\Pi_Y((x,y))=y$.

\begin{defn}\label{deflength}
We fix two points $x_0\in X$ and $y_0\in Y$, and  
a warping function $\phii \colon X\to [1,\infty)$.
Given a
curve $\gamma:[a,b]\to X\times Y$, we set $\gamma_X=\Pi_X\circ\gamma$ to be the projection of $\gamma$ to $X$ and
$\gamma_Y=\Pi_Y\circ\gamma$ to be the projection of $\gamma$ to $Y$. Thus
$\gamma=(\gamma_{X}, \gamma_{Y} )\colon [a,b]\to X\times Y$. Given such a curve $\gamma$ 
with $\gamma_X$ absolutely continuous with respect to the metric $d_X$ and $\gamma_Y$ absolutely continuous
with respect to $d_Y$, we set
\begin{equation*}
\ell_{\phii}(\gamma):=\ell_{d_{X}}(\gamma_{X})+\int_{a}^{b} \phii(d_{X}(\gamma_{X}(t),x_{0}))\ |\gamma'_{Y}(t)|\, dt.
\end{equation*}
For $z,w\in X\times Y$ we define $d_\pip(z,w)$ by
\[
d_\pip(z,w):=\inf_\gamma\ell_\pip(\gamma),
\]
where the infimum is over all curves $\gamma$ in $X\times Y$ connecting $z$ to $w$ 
for which $\gamma_X$ and $\gamma_Y$ are absolutely continuous.
\end{defn}

\begin{remark}
From~\cite[Proposition~3.8]{KKSS}, applied to the function $\pip\circ d_X(\cdot,x_0)$, we 
know that when $\gamma$ is a curve in $X\times Y$, the length 
$\ell_{d_\pip}(\gamma)$ of $\gamma$ with respect to the metric $d_\pip$ is equal to $\ell_\pip(\gamma)$. 
\end{remark}

\begin{defn}\label{warpedspace}
We denote by $X\times_\pip Y$ the metric space $(X\times Y, d_\pip)$. We refer to this as a warped product space.
\end{defn}

Note that $X\times_\pip Y$ is a geodesic space because $\pip$ is continuous and positive,
$(X,d)$ is proper, and $Y$ is compact.

\begin{remark}\label{pipremark}
As $\pip\ge 1$, Condition~(b) in Assumptions \ref{ass} guarantees that $\pip(t)\ge e^{\alpha t}$ for $t\ge 0$.   
Moreover, as $\phii$ is injective, it has a well-defined inverse $\phii^{-1}:[\phii(0),\infty)\to[0,\infty)$.  
We denote by $\overline{\phii^{-1}}$ the zero-extension of $\phii^{-1}$ to $[0,\infty)$.  That is, 
\begin{equation}\label{eq:phi inverse extension}
    \overline{\phii^{-1}}(t)=\begin{cases}
        0,&t\in[0,\phii(0))\\
        \phii^{-1}(t),&t\in[\phii(0),\infty).
    \end{cases}
\end{equation}	
\end{remark}

\section{Shapes of geodesics in warped product spaces}

As pointed out in the introduction, when $X$ is the Euclidean half line $[0,\infty)$ as in the setting of~\cite{KKSS}, the geodesic structure of
$X$ is very simple. In our setting, given a triple of points $x_0,x_1,x_2\in X$, the three points need not lie in the same geodesic, nor do
geodesics in $X$ need to be unique. Thus the identification of geodesics as in~\cite[Lemma~7.6]{KKSS} is not as simple, nor
need the formula for the Gromov product given in~\cite[Corollary~7.11]{KKSS} hold in our general setting. This section is
devoted to finding a description of geodesics in the warped space that are analogous to~\cite[Lemma~7.6]{KKSS}.

Curves joining points with the same $Y$-coordinate have a particularly simple structure. 
The following lemma follows directly from Definition~\ref{deflength}.

\begin{lem}\label{fixed y}
 If $\gamma$ is a $d_{\phii}$-geodesic from $(x_{1},y_{0})$ to $(x_{2},y_{0})$ for some fixed point $y_{0}\in Y$, then 
 $\gamma=(\gamma_{X},\{y_{0}\})$ where $\gamma_{X}$ is a geodesic in $X$.
\end{lem}

As a consequence of the above lemma, we have that if $X\times_\pip Y$ is Gromov hyperbolic, then we must also have that
$X\times\{y_0\}$ is also Gromov hyperbolic, that is, $X$ must also be Gromov hyperbolic. Thus we have a result that is
analogous to~\cite{AB2}; if $X\times_\pip Y$ is Gromov hyperbolic, then necessarily $X$ must be Gromov hyperbolic.
Recall here that $[\widetilde{x}, \widehat{x}]$ denote a choice of geodesic in 
$X$ connecting the pair of points $\widetilde{x}, \widehat{x}\in X$. A similar notation is used for geodesics in $Y$.

\begin{lem} \label{Right-angled curve}
 Let $(x_{1},y_{1}), (x_{2},y_{2})\in X\times Y$ with $d_{X}(x_{0},x_{1})\geq d_{X}(x_{0},x_{2})$, and let $\gamma$ be a path given by
 $\gamma=([x_1, x_2],\{y_{1}\}) * (\{x_{2}\},[y_1,y_2])$. 
 Let $\beta:[a,b]\to X\times Y$ be any 
    curve from $(x_{1},y_{1})$ to $(x_{2},y_{2})$, for which 
    $d_X(x_0,\beta_X(t))\geq d_X(x_0,x_2)$ 
    whenever $t\in[a,b]$. Then $\ell_{\varphi}(\gamma)\le \ell_{\varphi}(\beta)$. Moreover, if
    $\ell_{\varphi}(\gamma)=\ell_{\varphi}(\beta)$, then 
    \begin{itemize}
    \item $\beta_X$ is a (reparameterized) geodesic from $x_1$ to $x_2$,
    \item $\beta_Y$ is a (reparameterized) geodesic from $y_1$ to $y_2$, 
   \item for 
   $\mathcal{H}^1$-a.e.~$t\in[a,b]$
   for which $|\beta_Y^\prime(t)|>0$, we have that $d_X(\beta_X(t),x_0)=d_X(x_2,x_0)$.
   \end{itemize}
\end{lem}

\begin{proof}
Since $d_{X}(x_{0},\beta_X(t))\geq d_{X}(x_{0},x_{2})$ whenever $a\le t\le b$, it follows that
\begin{align*}
\ell_{d_X}(\beta_X)&\ge \ell_{d_X}(\gamma_X),\\
\int_a^b\pip(d_X(\beta_X(t),x_0)) |\beta_Y^\prime(t)|\, dt &\ge \pip(d_X(x_2,x_0))\ell_{d_Y}(\gamma_Y).
\end{align*}
Since $\ell_{\varphi}(\gamma)=\ell_{d_X}(\gamma_X)+\pip(d_X(x_2,x_0))\ell_{d_Y}(\gamma_Y)$ 
we have $\ell_{\varphi}(\gamma)\leq \ell_{\varphi}(\beta)$, which proves the first part of the statement.

Next suppose that $\beta_X$ is not a reparametrized geodesic from $x_1$ to $x_2$. Then 
$\ell_{d_X}(\beta_X)> \ell_{d_X}(\gamma_X)$, which implies that $\ell_{\varphi}(\gamma)<\ell_{\varphi}(\beta)$, 
and so we have a contradiction.
Similarly, suppose that 
$\beta_Y$ is not a reparameterized geodesic from $y_1$ to $y_2$. Then we have
\[
\int_a^b\pip(d_X(\beta_X(t),x_0)|\beta_Y^\prime(t)|\, dt \ge \pip(d_X(x_2,x_0))\ell_{d_Y}(\beta_Y)>\pip(d_X(x_2,x_0))\ell_{d_Y}(\gamma_Y),
\]
which again would contradict $\ell_{\varphi}(\gamma)=\ell_{\varphi}(\beta)$.

Finally, suppose that the collection of all $t\in[a,b]$  
for which the inequalities $|\beta_Y^\prime(t)|>0$
and $d_X(\beta_X(t),x_0)>d_X(x_2,x_0)$ both hold true has positive $\mathcal{H}^1$-measure. Then, 
\[
\int_a^b\pip(d_X(\beta_X(t),x_0)|\beta_Y^\prime(t)|\, dt > \pip(d_X(x_2,x_0))\ell_{d_Y}(\beta_Y)=\pip(d_X(x_2,x_0))\ell_{d_Y}(\gamma_Y),
\]
which would again lead to the contradiction $\ell_{\varphi}(\gamma)<\ell_{\varphi}(\beta)$.
\end{proof}

We next describe  geodesics connecting a general pair of points, see Figure~\ref{fig:geodesic X component} below.

\begin{figure}[h]
\centering
\includegraphics[scale=1]{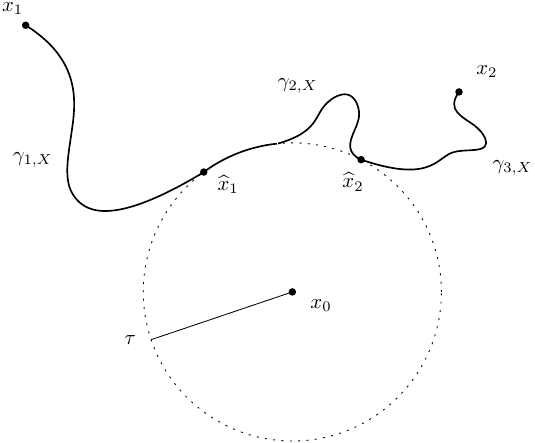}
\caption{{\small Here we show the possible structure of $\gamma_{X}$, when $\gamma$ is a geodesic connecting $(x_1,y_1)$ to $(x_2,y_2)$. By Proposition~\ref{lem:shape-phi-geods-Old}, we have that $\gamma_{1,X}$, $\gamma_{2,X}$, and $\gamma_{3,X}$ are geodesics in $X$, and it is possible for $\gamma_Y$ to be non-constant only at $t$ for which $d_X(\gamma_{2,X}(t),x_0)=\tau$.}\label{fig:geodesic X component}}
\end{figure}

\begin{prop}\label{lem:shape-phi-geods-Old}
Fix $(x_1,y_1), (x_2,y_2)\in X\times Y$ 
and a $d_\pip$-geodesic $\gamma$ from $(x_1,y_1)$ to $(x_2,y_2)$. Additionally define
\begin{itemize}
\item $\tau:=\inf\{d_X(\gamma_X(t),x_0)\}$,
\item $t_1=\inf\{t\, :\, d_X(\gamma_X(t),x_0)=\tau\}$ and $\widehat{x}_1=\gamma_X(t_1)$,
\item $t_2=\sup\{t\, :\, d_X(\gamma_X(t),x_0)=\tau\}$ and $\widehat{x}_2=\gamma_X(t_2)$.
\end{itemize}
Then there are $d_\pip$-geodesics $\gamma_1$, $\gamma_2$, and $\gamma_3$ such that
$\gamma=\gamma_1*\gamma_2*\gamma_3$, where $\gamma_1=([x_1,\widehat{x}_1],\{y_1\})$,
$\gamma_3=([\widehat{x}_2,x_2],\{y_2\})$, and $\gamma_2$ 
has the following structure:
\[
\gamma_{2,X}=[\widehat{x}_1,\widehat{x}_2],\qquad \gamma_{2,Y}=[y_1,y_2],
\]
and $|\gamma_{2,Y}^\prime(t)|=0$ whenever $d_X(\gamma_{2,X}(t),x_0)>\tau$.
\end{prop}

\begin{proof}
Let $\gamma:[a,b]\to X\times Y$ be the $d_\pip$-geodesic fixed in the hypothesis of the lemma.
By Lemma~\ref{Right-angled curve}, applied with $\gamma_1:=\gamma\vert_{[a,t_1]}$ playing the role of $\beta$, we know that
$d_X(\gamma_{1,X}(t),x_0)>\tau$ when $a\le t<t_1$ and so $\gamma_{1,Y}(t)=y_1$ for each $t\in [a,t_1]$. Similarly, 
by inverting the direction of the curve,
with $\gamma_3=\gamma\vert_{[t_2,b]}$ we have that $\gamma_{3,Y}(t)=y_2$ for each $t\in[t_2,b]$.
With $\gamma_2=\gamma\vert_{[t_1,t_2]}$, we can apply Lemma~\ref{Right-angled curve} again to conclude the structure
of $\gamma_2$ specified in the lemma. This concludes the proof of the lemma.
\end{proof}

In the next section we are only interested in the computation of the $d_\pip$--distances, and for this purpose we will consider
simple choices of $d_\pip$-geodesics identified in the following corollary.

\begin{cor}\label{lem:shape-phi-geods}
Let $(x_1,y_1)$ and $(x_2,y_2)$ be two points in $X\times Y$, and $\gamma$ be a $d_\pip$--geodesic from $(x_1,y_1)$ to
$(x_2,y_2)$, and let $\widehat{x}$ be either $\widehat{x}_1$ or $\widehat{x}_2$, with 
$\widehat{x}_1$ and $\widehat{x}_2$ given by Proposition~\ref{lem:shape-phi-geods-Old}.
Then the curve
$\pa$ in $X\times Y$, given by
\[
\pa=([x_1,\widehat{x}],\{y_1\})*(\{\widehat{x}\},[y_1,y_2])*([\widehat{x},x_2],\{y_2\})
\]
is a $d_\pip$--geodesic from $(x_1,y_1)$ to $(x_2,y_2)$.
\end{cor}

\begin{proof}
The claim follows from Proposition~\ref{lem:shape-phi-geods-Old} and the fact that $\ell_\pip(\pa)\le \ell_\pip(\gamma)$.
\end{proof}

\section{Construction of the tree $T(W)$ and estimates for warped metrics}

In this section we relate warped product distances to distances in a suitable tree, which will be useful in our later work. 
 In a tree $T$, we  denote geodesics with the subscript $T$, the metric by $h$, and the distance between a point 
 $\xi\in T$ and a set $E\subset T$ by $\dist_T(\xi, E)$.

We fix $(x_1,y_1)$, $(x_2,y_2)$ in $X\times Y$, and then we fix $\widehat{x}\in X$ 
identified in Corollary~\ref{lem:shape-phi-geods},
so that
\begin{equation*} 
([x_1,\widehat{x}],\{y_1\})*(\{\widehat{x}\},[y_1,y_2])*([\widehat{x},x_2],\{y_2\})
\end{equation*}
is a $d_\pip$--geodesic from $(x_1,y_1)$ to $(x_2,y_2)$. Let $\gamma_{1}, \gamma_{2}, \gamma_{3}$ be geodesics in 
$X$ from $x_{0}$ to $x_{1}, x_{2}, \widehat{x}$ respectively.
Let $W$ be the union of these three geodesics and let $T(W)$ be a corresponding tree with $u\colon W\to T(W)$ the continuous 
map given by Corollary~\ref{thm:tree}.

With the tree established, we choose points $P,Q$ as follows: 
\begin{itemize}
\item The geodesic $[u(x_0),u(\widehat{x})]_T$ in $T(W)$ has a common 
segment with $[u(x_0),u(x_1)]_T$ and a common segment with $[u(x_0),u(x_2)]_T$ (though each could be a single point). 
We choose $P \in T(W)$ such that $[u(x_0),P]_T$ is the largest of these common segments.
\item Choose $Q\in T(W)$ such that $[u(x_0),Q]_T$ is the largest common segment
of the two geodesics $[u(x_0),u(x_1)]_T$ and $[u(x_0),u(x_2)]_T$. 
\end{itemize}

\begin{figure}[h]
\centering
  \includegraphics[scale=.6]{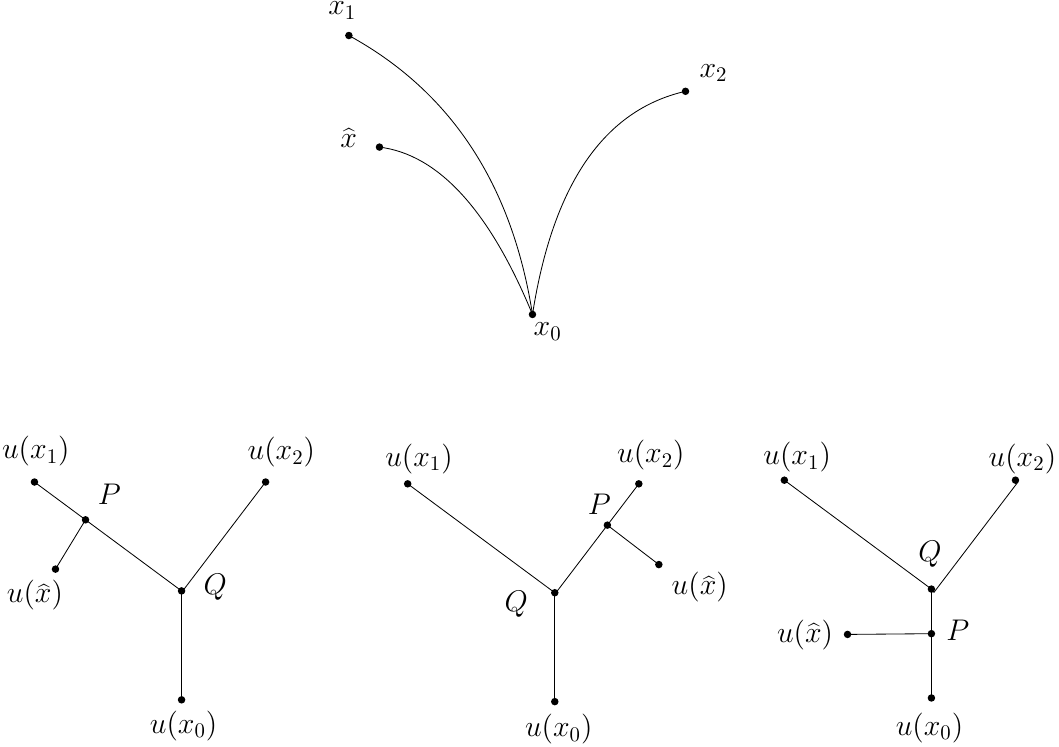}
\caption{{\small $W$ and the three configurations of $T(W)$}}\label{fig:Z and T(Z)}
\end{figure}

Notice that we have three possible configurations; either $P$ is in $[Q,u(x_1)]_T$, or $P$ is in $[Q,u(x_2)]_T$, or $P$ is in $[u(x_0),Q]_T$,
see Figure~\ref{fig:Z and T(Z)}. From Corollary~\ref{thm:tree}, we also recall that
\begin{enumerate}
\item $T(W)$ is equipped with the natural length metric $h$.
\item The restriction of $u$ to $\gamma_i$ is an isometry for each $i=1,\cdots, 3$.
\item For each $x,y\in W$ we have that 
\begin{equation}\label{eq:tree-metric2dx-metric}
 d(x,y)-6\delta\le h(u(x),u(y))\le d(x,y).
\end{equation}
\end{enumerate}

Using Corollary~\ref{thm:tree}, we  prove the following. 

\begin{lem}\label{lem:dis-between-hatx-geods}
Let $(x_1,y_1)$, $(x_2,y_2)$ in $X\times Y$ and $\widehat{x}\in X$ be as in the description above. 
Then 
\[
\min\{\dist_X(\widehat{x},\gamma_1),\, \dist_X(\widehat{x},\gamma_2)\}\le 16\delta
\]
and 
\[
h(u(\widehat{x}),P)\le 10\delta.
\]
\end{lem}

\begin{proof}
Without loss of generality we can assume that $\delta>0$.
We treat the three configurations, described in the above preamble, separately.

\vskip .3cm
\noindent{\bf Case~(a):} Suppose that $P$ is in $[Q,u(x_1)]_T$. The goal is to show that $\dist_X(\widehat{x},\gamma_1)\le 16\delta$. 
If $h(P,u(\widehat{x}))<10\delta$, then the claim will follow from~\eqref{eq:tree-metric2dx-metric}.
So suppose that $h(P,u(\widehat{x}))\ge 10\delta$. Let $p_1\in\gamma_1$ and $\widehat{p_1}\in\gamma_3$ such that
$u(p_1)=u(\widehat{p_1})=P$; then $d_X(p_1,\widehat{p_1})-6\delta\le h(P,P)=0$, and so $d_X(p_1,\widehat{p_1})\le 6\delta$.
Hence, as $d_X(\widehat{p_1},\widehat{x})=h(u(\widehat{p_1}),u(\widehat{x}))$,
\begin{align*}
d_X(x_0,\widehat{x})=d_X(x_0,\widehat{p_1})+d_X(\widehat{p_1},\widehat{x})
  \ge d_X(x_0,\widehat{p_1})+10\delta
  &\ge 10\delta+d_X(x_0,p_1)-d_X(p_1,\widehat{p_1})\\
  &\ge 4\delta+d_X(x_0,p_1).
\end{align*}
Since $\pip$ is an increasing function, it follows that 
\[
\pip(d_X(x_0,\widehat{x}))\ge \varphi(d_{X}(x_{0},p_{1})).
\]
Now, by~\eqref{eq:tree-metric2dx-metric}, Corollary~\ref{lem:shape-phi-geods}, 
and the assumption that $h(P,u(\widehat{x}))\ge 10\delta$, we have
\begin{align*}
d_\pip((x_1,y_1),(x_2,y_2))&=d_X(x_1,\widehat{x})+d_X(x_2,\widehat{x})+\pip(d_X(x_0,\widehat{x}))\, d_Y(y_1,y_2)\\
 &\ge h(u(x_1),u(\widehat{x}))+h(u(x_2),u(\widehat{x}))+\pip(d_X(x_0,p_1))\, d_Y(y_1,y_2)\\
&=2h(u(\widehat x),P)+h(u(x_1),u(x_2))+\pip(d_X(x_0,p_1))\, d_Y(y_1,y_2)\\
&\ge 20\delta +h(u(x_1),u(x_2))+\pip(d_X(x_0,p_1))\, d_Y(y_1,y_2).
\end{align*}
Let $q_1\in\gamma_1$ such that $u(q_1)=Q$.
Again using~\eqref{eq:tree-metric2dx-metric}, as well as the assumption that $P$ is in $[Q,u(x_1)]_T$, we see that 
\begin{align*}
d_\pip((x_1,y_1),(x_2,y_2))
&\ge 20\delta +h(u(x_1),u(x_2))+\pip(d_X(x_0,p_1))\, d_Y(y_1,y_2)\\
 &=20\delta+h(u(x_1),Q)+h(u(x_2),Q)+\pip(d_X(x_0,p_1))\, d_Y(y_1,y_2)\\
 & \ge 20\delta+d_X(x_1,q_1)+d_X(x_2,q_1)-6\delta\, +\pip(d_X(x_0,p_1))\, d_Y(y_1,y_2)\\
 & \ge 14\delta+d_X(x_1,q_1)+d_X(x_2,q_1)+\pip(d_X(x_0,q_1))\, d_Y(y_1,y_2)\\
 &>d_X(x_1,q_1)+d_X(x_2,q_1)+\pip(d_X(x_0,q_1))\, d_Y(y_1,y_2)\\
 &=\ell_\pip(\beta),
\end{align*}
where $\beta=([x_1,q_1],\{y_1\})*(\{q_1\},[y_1,y_2])*([q_1,x_2],\{y_2\})$ is a path connecting $(x_1,y_1)$ to $(x_2,y_2)$,
which violates the definition of the metric $d_\pip$. Therefore we must have that $h(P,u(\widehat{x}))\le 10\delta$, and the conclusion 
follows.

\vskip .3cm
\noindent{\bf Case~(b):} The point $P$ is in $[Q,u(x_2)]_T$. This case follows from the argument for Case~(a) mutatis mutandis. 

\vskip .3cm
\noindent{\bf Case~(c):} The point $P$ is in $[u(x_0),Q]_T$. Suppose 
for contradiction that $h(u(\widehat{x}),P)\ge 6\delta$.
 Let $p_1\in\gamma_1$ such that $u(p_1)=P$. 
Then
\begin{align*}
h(u(x_1),u(\widehat{x}))+h(u(x_2),u(\widehat{x}))&=h(u(x_1),P)+h(u(x_2),P)+2\, h(u(\widehat{x}),P)\\
&\ge h(u(x_1),P)+h(u(x_2),P)+12\delta\\
&\ge d_X(x_1,p_1)+d_X(x_2,p_1)-6\delta+12\delta.
\end{align*}
Moreover, we have
\begin{align*}
d_X(\widehat{x},x_0)=h(u(\widehat{x}),P)+h(P,u(x_0))\ge 6\delta+h(P,u(x_0))= 6 \delta+d_X(p_1,x_0).
\end{align*}
From Assumptions~\ref{ass}, it follows that 
\[
\pip(d_X(\widehat{x},x_0))\ge \pip(d_X(p_1,x_0)).
\]
From Corollary~\ref{lem:shape-phi-geods} and~\eqref{eq:tree-metric2dx-metric}, we see that 
\begin{align*}
d_\pip((x_1,y_1),(x_2,y_2))&=d_X(x_1,\widehat{x})+d_X(x_2,\widehat{x})+\pip(d_X(\widehat{x},x_0)\, d_Y(y_1,y_2)\\
&\ge h(u(x_1),u(\widehat{x}))+h(u(\widehat{x}), u(x_2))+\pip(h(u(\widehat{x}),u(x_0)))\, d_Y(y_1,y_2)\\
 &  \ge h(u(x_1),P)+6\delta+h(u(x_2),P)+6\delta+\pip(h(P,u(x_0)))\, d_Y(y_1,y_2)\\
&\ge d_X(x_1,p_1)+6\delta+\left[d_X(x_2,p_1)-6\delta\right]+6\delta+\pip(d_X(p_1,x_0))\, d_Y(y_1,y_2)\\
&>d_X(x_1,p_1)+d_X(x_2,p_1)+\pip(d_X(x_0,p_1))\, d_Y(y_1,y_2).
\end{align*}
This leads to a contradiction, as with $\beta:=([x_1,p_1],\{y_1\})*(\{p_1\},[y_1,y_2])*([p_1,x_2],\{y_2\})$, the above inequality shows 
$\ell_\pip(\beta)<d_\pip(x_1,y_1),(x_2,y_2))$ despite the fact that $\beta$ joins $(x_{1},y_{1})$ and $(x_{2}, y_{2})$.
Hence we must have $h(u(\widehat{x}),P)\le 6\delta$, and so $d_X(\widehat{x},p_1)\le 12\delta$, yielding the conclusion
in this case.
\end{proof}

\begin{lem}\label{lem:Q2P} 
Let $(x_1,y_1)$, $(x_2,y_2)$ in $X\times Y$, $\widehat{x}\in X$, and $Q$ be as in the description at the beginning of this section.
Then
\begin{equation}\label{eq:HoHum}
\dist_X(\widehat{x},\gamma_1)\leq 34\delta \text{~and~} \dist_X(\widehat{x},\gamma_2)\leq 34\delta.
\end{equation}
Moreover, 
\begin{equation}\label{eq:hat-x-Q}
\dist_T(u(\widehat{x}),[Q,u(x_0)]_T))\leq 40\delta.
\end{equation}
\end{lem}

\begin{proof}
By Lemma \ref{lem:dis-between-hatx-geods},
without loss of generality we can assume that 
\begin{equation}\label{wlog16}
\dist_X(\widehat{x},\gamma_1)\leq 16\delta.
\end{equation}
With this assumption,
%Then 
it suffices to show that $\dist_X(\widehat{x},\gamma_2)\leq 34\delta$ and verify~\eqref{eq:hat-x-Q}.

\vskip .3cm
\noindent{\bf Case~(1):}
Suppose $P$ lies on the geodesic $[u(x_2),u(x_0)]_T$. 
This corresponds to Cases~(b) and~(c) of the proof of Lemma~\ref{lem:dis-between-hatx-geods}.
Note that 
$h(u(\widehat{x}),P)\le 10\delta$ by Lemma~\ref{lem:dis-between-hatx-geods}.
It follows that
$\dist_X(\widehat{x},\gamma_2)\le  h(u(\widehat{x}),P)+6\delta\le 16\delta$, 
which yields~\eqref{eq:HoHum}.
Moreover, by \eqref{wlog16}, if $P$ lies in the segment $[u(x_2),Q]_T$, then we must have 
\[
h(P,Q)\le \dist_T(u(\widehat{x}),[u(x_1),u(x_0)]_T)\le 16\delta.
\]
Hence we have the validity of~\eqref{eq:hat-x-Q}. 

\vskip .3cm
\noindent{\bf Case~(2):} Suppose $P$ lies in the geodesic $[u(x_1),Q]_T$,
which corresponds to Case~(a) of the proof of  
Lemma~\ref{lem:dis-between-hatx-geods}.
We fix a geodesic $[x_1,x_2]$ in $X$ joining $x_1$ to $x_2$, and let $\overline{x}\in [x_1,x_2]$ such that 
$\dist_X(x_0,[x_1,x_2])=d_X(x_0,\overline{x})$. 
Moreover, if $d_X(x_0,\widehat{x})>d_X(x_0,\overline{x})$,
then by the strict monotonicity of $\pip$ we must have that 
\begin{align*}
d_X(x_1,\widehat{x})+d_X(x_2,\widehat{x})+\pip(d_X(x_0,\widehat{x}))&d_Y(y_1,y_2)\\
&>d_X(x_1,\overline{x})+d_X(x_2,\overline{x})+\pip(d_X(x_0,\overline{x}))d_Y(y_1,y_2),
\end{align*}
which would be a contradiction of the length minimization property that defined $d_\pip$. It follows that we must have
\begin{equation}\label{eq:x-bar2x-hat}
d_X(x_0,\widehat{x})\le d_X(x_0,\overline{x}).
\end{equation}

From the definition of Gromov product we know that $0\leq (x_1|x_2)_{x_0}\leq d_{X}(x_{1},x_{0})$. Hence we may choose
$x_1^*\in\gamma_1$ such that $d_X(x_0,x_1^*)=(x_1|x_2)_{x_0}$. Let $p_1\in\gamma_1$ be as in the proof 
of Lemma~\ref{lem:dis-between-hatx-geods}.
We know from Lemma~\ref{lem:dis-between-hatx-geods}  that 
\begin{equation}\label{eq:hate-p}
d_X(\widehat{x},p_1)\le h(u(\widehat{x}), P)+6\delta\le 16\delta. 
\end{equation}
Since $d_X(x_1,x_0)\ge d_X(\overline{x},x_0)$, we can choose $\overline{x_1}\in \gamma_1$ 
such that $d_X(\overline{x_1},x_0)=d_X(x_0,\overline{x})$. 
By~\cite[page~197]{VaisalaExpo}, 
we know that $d_X(\overline{x_1},x_1^*)\le 2\delta$, and 
by~\cite[Tripod Lemma~2.15]{VaisalaExpo}, we also know
that $\dist_X(x_1^*,\gamma_2)\le 4\delta$. 
By~\eqref{eq:x-bar2x-hat}, we have
$d_X(\widehat{x},x_0)\le d_X(\overline{x},x_0)=d_X(\overline{x_1},x_0)$, 
and so it follows that
\[
h(P,u(x_0))\le h(u(\widehat{x}),u(x_0))\le d_X(\widehat{x},x_0)\le d_X(\overline{x_1},x_0)\le h(u(\overline{x_1}),u(x_0))+6\delta.
\]
Now, if $h(P,u(x_0))>h(u(\overline{x_1}),u(x_0))$, then $h(u(\overline{x_1}),P)\le 6\delta$, from which
it follows that 
\[
h(u(\widehat{x}), u(\overline{x_1}))\le 10\delta+6\delta=16\delta,
\]
where we also used Lemma~\ref{lem:dis-between-hatx-geods}. 
Hence
$d_X(\widehat{x},\overline{x_1})\le 28\delta$. Then
\[
\dist_X(\widehat{x},\gamma_2)\leq d_X(\widehat{x},\overline{x_1})+d_X(\overline{x_1},x_1^*)
  +\dist_X(x_1^*,\gamma_2)\leq 28\delta+2\delta+4\delta=34\delta,
\]
which proves~\eqref{eq:HoHum}.
If $h(P,u(x_0))\le h(u(\overline{x_1}),u(x_0))$, then we must necessarily have that
\[
\dist_T(P,[u(x_0),u(x_2)]_T)\le \dist_T(u(\overline{x_1}),[u(x_0),u(x_2)]_T).
\]
On the other hand, as observed above, $\dist_X(x_1^*,\gamma_2)\le 4\delta$ and $d_X(x_1^*,\overline{x_1})\le 2\delta$; therefore
$\dist_X(\overline{x_1},\gamma_2)\le 6\delta$, whence we obtain that 
\[
\dist_T(P,[u(x_0),u(x_2)]_T)\le \dist_T(u(\overline{x_1}),[u(x_0),u(x_2)]_T)\le \dist_X(\overline{x_1},\gamma_2)\le 6\delta.
\]
It follows from~\eqref{eq:hate-p} that
\[
\dist_X(\widehat{x},\gamma_2)\le d_X(\widehat{x},p_1)+\dist_X(p_1,\gamma_2)\le 16\delta+\dist_T(P,[u(x_0),u(x_2)]_T)+6\delta\le 28\delta,
\]
which again verifies~\eqref{eq:HoHum}.

In both the cases $h(P,u(x_0))>h(u(\overline{x_1}),u(x_0))$ and $h(P,u(x_0))\le h(u(\overline{x_1}),u(x_0))$, we get
$\dist_T(u(\widehat{x}),[Q,u(x_0)]_T)=\dist_T(u(\widehat{x}),[u(x_0),u(x_2)]_T)])\leq 34\delta+6\delta=40\delta$.
\end{proof}

Now we have sufficient tools to estimate the metric $d_\pip$ in terms of the metric on the corresponding tree $T(W)$
considered above.

\begin{prop}\label{prop:Product-Estimate}
Let $(x_1,y_1),(x_2,y_2)\in X\times Y$ be as in the description at the beginning of the section. 
Suppose that $P\in[u(x_0),u(x_1)]_T$.
For $t\in [0, (x_1|x_2)_{x_0}+40\delta]$, let $p_t\in [u(x_0),u(x_1)]_T$ such that $h(u(x_0),p_t)=t$
when $t\le d_X(x_0,x_1)$, and $p_t=u(x_1)$ if $d_X(x_0,x_1)<t\le  (x_1|x_2)_{x_0}+40\delta$.
We set
\[
F(t)=h(u(x_1),p_t)+h(u(x_2),p_t)+\pip(t)\, d_Y(y_1,y_2).
\]
Let $t_{12}\in [0, (x_1|x_2)_{x_0}+40\delta]$ be such that $F(t_{12})\le F(t)$ for 
all $t\in [0, (x_1|x_2)_{x_0}+40\delta]$.
Then 
\begin{equation*} %\label{eq:comp-Fvsdpip}
F(t_{12})\le d_\pip((x_1,y_1),(x_2,y_2))\le F(t_{12})+6\delta.
\end{equation*}
\end{prop}

Note that by relabelling points we can reduce a general situation to the case $P \in [u(x_0), u(x_1)]_T$ and hence make use of 
Proposition~\ref{prop:Product-Estimate}. We will apply this strategy in Section \ref{Sec:ProofMain1}.

\begin{proof}
We choose $t_{12}^*\in\gamma_1$ 
such that $u(t_{12}^*)=p_{t_{12}}$, and
consider the path $\gamma_{12}:=([x_1,t_{12}^*],\{y_1\})*(\{t_{12}^*\}, [y_1,y_2]_Y)*([t_{12}^*,x_2],\{y_2\})$. We know that
\begin{align*}
d_\pip((x_1,y_1),(x_2,y_2))&\le \ell_{\pip}(\gamma_{1,2})\\
&=d_X(x_1,t_{12}^*)+d_X(x_2,t_{12}^*)+\pip(d_X(t_{12}^*,x_0))\, d_Y(y_1,y_2).
\end{align*}
Hence, by~\eqref{eq:tree-metric2dx-metric}, we obtain
\begin{align*}
d_\pip((x_1,y_1),(x_2,y_2))&\le h(u(x_1),p_{t_{12}})+[h(u(x_2),p_{t_{12}})+6\delta]+\pip(t_{12})\, d_Y(y_1,y_2)\\
 &=F(t_{12})+6\delta.
\end{align*}
By~\eqref{eq:gromov-prod2xbar} we have that
\[
2\delta+(x_1|x_2)_{x_0}\ge d_X(x_0,\overline{x}),
\]
where $\overline{x}$ is a point on the geodesic $[x_1,x_2]$ that is closest to $x_0$.
Moreover, if $d_X(x_0,\widehat{x})>d_X(x_0,\overline{x})$,
then by the strict monotonicity of $\pip$ we must have that 
\begin{align*}
d_X(x_1,\widehat{x})+d_X(x_2,\widehat{x})+\pip(d_X(x_0,\widehat{x}))&d_Y(y_1,y_2)\\
&>d_X(x_1,\overline{x})+d_X(x_2,\overline{x})+\pip(d_X(x_0,\overline{x}))d_Y(y_1,y_2),
\end{align*}
which would be a contradiction of the length minimization property that defined $d_\pip$. It follows that we must have
\begin{equation*}
d_X(x_0,\widehat{x})\le d_X(x_0,\overline{x}).
\end{equation*}
Now we have 
\[
2\delta+(x_1|x_2)_{x_0}\ge d_X(x_0,\widehat{x}).
\]
Note by Lemma~\ref{lem:Q2P} that $\dist_T(P,[Q,u(x_0)]_T)\le 34\delta$.
Thus, by~\eqref{eq:tree-metric2dx-metric}, we have 
\[
2\delta+(x_1|x_2)_{x_0}\ge h(u(x_0),u(\widehat{x}))\ge h(u(x_0),Q)\ge h(u(x_0),P)-34\delta.
\]
Now, by~\eqref{eq:tree-metric2dx-metric}, 
the minimality of $F(t_{12})$ and the above discussion, it follows that 
\begin{align*}
F(t_{12})&\le h(u(x_1), P)+h(u(x_2), P)+\pip(h(P, u(x_0)))\, d_Y(y_1,y_2)\\
  &\le h(u(x_1),u(\widehat{x}))+h(u(x_2),u(\widehat{x}))+\pip(h(u(\widehat{x}),u(x_0)))\, d_Y(y_1,y_2)\\
  &\le d_X(x_1,\widehat{x})+d_X(x_2,\widehat{x})+\pip(d_X(\widehat{x},x_0))\, d_Y(y_1,y_2)\\
  &=d_\pip((x_1,y_1),(x_2,y_2)),
\end{align*}
where we also used the fact that $\pip$ is monotone increasing. This completes the proof.
\end{proof}

\section{Gromov hyperbolicity: proof of Theorem~\ref{thm:Main1}}\label{Sec:ProofMain1}

The goal of this section is to prove Theorem~\ref{thm:Main1}. To do so, we first establish a 
relationship between Gromov products and maximizers of a functional involving $\varphi$. 
Recall the choice of $x_0\in X$ made in Assumption~\ref{ass}, used to construct the metric $d_\pip$. We now also fix
a companion point $y_0\in Y$ and consider Gromov products based at the point $z_0=(x_0,y_0)\in X\times Y$.
We are now ready to prove that $X\times_\pip Y$ is Gromov hyperbolic. To do so, we need estimates for the Gromov product
$(z_1|z_2)_{z_0}$ where $z_1=(x_1,y_1), z_2=(x_2,y_2)\in X\times Y$. 
In what follows
we define the function
\begin{equation*}
G[z_1,z_2]:[0,(x_1|x_2)_{x_0}+40\delta]\to\R 
\end{equation*}
by
\begin{equation}\label{eq:Gzz}
G[z_1,z_2](t)=2t-\pip(t)\, d_Y(y_1,y_2).
\end{equation}
 In the rest of the paper we set 
 \begin{equation*} 
 \Delta:=2\pip(0)\, \diam(Y).
 \end{equation*}
The following lemma is the extension of~\cite[Corollary~7.11]{KKSS} to our, more general, setting.

\begin{lem}\label{lem:product}
Let $z_1,z_2\in X\times Y$, where $z_1=(x_1,y_1)$, $z_2=(x_2,y_2)$. Then
\begin{align*}
\max_{[0,(x_1|x_2)_{x_0}+40\delta]}\ G[z_1,z_2](t)-86\delta \le 2(z_1|z_2)_{z_0}
  \le \Delta+\max_{[0,(x_1|x_2)_{x_0}+40\delta]}\, G[z_1,z_2](t).
\end{align*}
\end{lem}

\begin{proof}
Since the two functions $G[z_1,z_2]$ and $G[z_2,z_1]$ are the same,
as pointed out in the previous section we may choose the labeling of $(x_1,y_1)$ and $(x_2,y_2)$ so that $P\in[u(x_0), u(x_1)]_T$.
By definition,
\[
2(z_1|z_2)_{z_0}=d_\pip(z_1,z_0)+d_\pip(z_2,z_0)-d_\pip(z_1,z_2).
\]
We have
\begin{align*}
d_\pip(z_1,z_0)&=d_X(x_1,x_0)+\pip(0)d_Y(y_1,y_0)=h(u(x_1),u(x_0))+\pip(0)d_Y(y_1,y_0),\\
d_\pip(z_2,z_0)&=d_X(x_2,x_0)+\pip(0)d_Y(y_2,y_0)=h(u(x_2),u(x_0))+\pip(0)d_Y(y_2,y_0).
\end{align*}
Hence by Proposition~\ref{prop:Product-Estimate},
\begin{align}\label{eq:A1}
2(z_1|z_2)_{z_0}&=h(u(x_1),u(x_0))+h(u(x_2),u(x_0))+\pip(0)[d_Y(y_1,y_0)+d_Y(y_2,y_0)]-d_\pip(z_1,z_2)\notag\\
&\le \Delta+h(u(x_1),u(x_0))+h(u(x_2),u(x_0))-F(t_{12}),
\end{align}
and
\begin{equation}\label{eq:A2}
2(z_1|z_2)_{z_0}\ge h(u(x_1),u(x_0))+h(u(x_2),u(x_0))-F(t_{12})-6\delta.
\end{equation}
From the choice of $t_{12}$ in Proposition~\ref{prop:Product-Estimate}, we have that
\[
F(t_{12})=\min_{[0,(x_1|x_2)_{x_0}+40\delta]}\, (h(u(x_1),u(x_0))-t)_++h(u(x_2),p_t)+\pip(t)\, d_Y(y_1,y_2).
\]
Since $t_{12}\le (x_1|x_2)_{x_0}+40\delta$, it follows that
\begin{align*}
t_{12}&\le \frac12 [d_X(x_1,x_0)+x_X(x_2,x_0)-d_X(x_1,x_2)]+40\delta\\
 &\le \frac12 [h(u(x_1),u(x_0))+h(u(x_2),u(x_0))-h(u(x_1),u(x_2))]+40\delta\\
 &=h(u(x_0),Q)+40\delta.
\end{align*}
Therefore, $\dist_T(p_{t_{12}},[Q,u(x_0)]_T)\le 40\delta$. 
A mutatis mutandis argument to the above also gives that when $0\le t\le (x_1|x_2)_{x_0}+40\delta$, we
have $t\le h(u(x_0),Q)+40\delta$, and hence
\begin{equation}\label{eq:pendant}
\dist_T(p_{t},[Q,u(x_0)]_T)\le 40\delta.
\end{equation}
Denote
\[
\tau:=\min\{d_X(x_1,x_0),\, (x_1|x_2)_{x_0}+40\delta\}.
\] 
For $0\le t\le \tau$, we have
\[
F(t)=h(u(x_1),u(x_0))-t+h(u(x_2),p_t)+\pip(t)\, d_Y(y_1,y_2).
\]
When $p_t\in[Q,u(x_0)]_T$ we have that
\[
h(u(x_2),p_t)=h(u(x_2),u(x_0))-t,
\]
and so 
\[
F(t)=h(u(x_1),u(x_0))+h(u(x_2),u(x_0))-2t+\pip(t)\, d_Y(y_1,y_2).
\]
When $p_t\in[u(x_1),Q]_T$, we have
\[
h(u(x_2),p_t)=h(u(x_2),u(x_0))-h(u(x_0),Q)+h(Q,p_t)=h(u(x_2),u(x_0))-t+2h(Q,p_t).
\]
Thus, in this case, we have by~\eqref{eq:pendant},
\begin{align*}
F(t)&=h(u(x_1),u(x_0))+h(u(x_2),u(x_0))-2t+\pip(t)\, d_Y(y_1,y_2)+2h(Q,p_t)\\
&\le h(u(x_1),u(x_0))+h(u(x_2),u(x_0))-2t+\pip(t)\, d_Y(y_1,y_2)+80\delta.
\end{align*}
Now it follows that
\begin{align*}
F(t_{12})&\ge \min_{[0,\tau]}\, h(u(x_1),u(x_0))+h(u(x_2),u(x_0))-2t+\pip(t)\, d_Y(y_1,y_2),\\
F(t_{12})&\le \min_{[0,\tau]}\, h(u(x_1),u(x_0))+h(u(x_2),u(x_0))-2t+\pip(t)\, d_Y(y_1,y_2)+80\delta.
\end{align*}
It follows from~\eqref{eq:A1} and~\eqref{eq:A2} that
\begin{align*}
2(z_1|z_2)_{z_0} &\le\Delta -\,\min_{[0,\tau]}\, [\pip(t)\, d_Y(y_1,y_2)-2t],\\
2(z_1|z_2)_{z_0} &\ge -6\delta-80\delta-\,\min_{[0,\tau]}\, [\pip(t)\, d_Y(y_1,y_2)-2t].
\end{align*}
The lemma now follows from the definition of $G$ in~\eqref{eq:Gzz}. 
\end{proof}

Now we are ready to prove Theorem~\ref{thm:Main1}. 

\begin{proof}[Proof of Theorem~\ref{thm:Main1}]
By Remark~\ref{baseswitch}, it suffices to show that
\begin{equation}\label{product-compare}
(z_1|z_2)_{z_0}\geq \min\{(z_1|z_3)_{z_0}, (z_2|z_3)_{z_0}\} - \left(\tfrac34\Delta+\alpha^{-1}+44\delta\right)
\end{equation}
for all $z_1=(x_1,y_1), z_2=(x_2,y_2), z_3=(x_3,y_3)\in X\times Y$.  For all indices $i,j\in\{1, 2, 3\}$ with $i<j$,  
there is $t_{ij}\in [0,(x_i|x_j)_{x_0}+40\delta]$ 
such that $G[z_i,z_j]$ attains its maximum value. 
By Lemma~\ref{lem:product}, to prove (\ref{product-compare}), we only need to show that
\[
G[z_1,z_2](t_{12})\geq \min\{ G[z_1,z_3](t_{13}),  G[z_2,z_3](t_{23})\}-(2\delta+2\alpha^{-1}+\Delta/2).
\]
By symmetry, we can assume that $t_{13}\le t_{23}$, by switching the labelling of $z_1$ and $z_2$ if necessary. 
Since $t_{13}\in [0,(x_1|x_3)_{x_0}+40\delta]$ and $t_{23}\in[0,(x_2|x_3)_{x_0}+40\delta]$,  
it follows that
$t_{13}\le (x_1|x_3)_{x_0}+40\delta$ and $t_{13}\le t_{23}\le (x_2|x_3)_{x_0}+40\delta$. 
Since $X$ is $\delta$-hyperbolic,
\begin{equation}\label{t13est}
t_{13}\le \min\{(x_1|x_3)_{x_0}, (x_2|x_3)_{x_0}\}+40\delta\le (x_1|x_2)_{x_0}+41\delta.
\end{equation}
In what follows we extend the function $G[z_1,z_2]$ to the interval $[0, (x_1|x_2)_{x_0}+41\delta]$ using the same formula
\[
G[z_1,z_2](t)=2t-\pip(t)\, d_Y(y_1,y_2).
\]
Since $\phii$ is increasing, 
for $t\in[(x_1|x_2)_{x_0}+40\delta, (x_1|x_2)_{x_0}+41\delta]$ we have 
\begin{align*}
G[z_1,z_2](t)&\leq 2((x_1|x_2)_{x_0}+41\delta)-\phii((x_1|x_2)_{x_0}+40\delta)d_Y(y_1,y_2)\\
&= G[z_1,z_2]((x_1|x_2)_{x_0}+40\delta)+2\delta.
\end{align*}
Hence, 
\[
G[z_1,z_2](t_{12})\geq \max_{[0,(x_1|x_2)_{x_0}+41\delta]}\ G[z_1,z_2](t)-2\delta.
\]
Combining this with~\eqref{t13est} yields $G[z_1,z_2](t_{12})\geq G[z_1,z_2](t_{13})-2\delta$.
By the triangle inequality,
\begin{align}
G[z_1,z_2](t_{12}) &\ge G[z_1,z_2](t_{13})-2\delta \nonumber \\
&=2t_{13}-\pip(t_{13})d_Y(y_1, y_2)-2\delta \nonumber \\
&\ge 2t_{13}-\pip(t_{13})(d_Y(y_1, y_3)+d_Y(y_2, y_3))-2\delta \nonumber  \\
&=G[z_1,z_3](t_{13})-\pip(t_{13})d_Y(y_2, y_3)-2\delta. \label{G12} 
\end{align}
For each $i<j$ we have that $G[z_i,z_j]^\prime(t)=2-\phii^\prime(t)d_Y(y_i,y_j)$. Since the function $G[z_i,z_j]$ is 
continuous and $[0, (x_1|x_2)_{x_0}+41\delta]$
is compact, the function achieves its maximum in this interval at $t_{ij}$ with either $t_{ij}=0$ or else
$G[z_i,z_j]^\prime(t_{ij})\ge 0$. Hence
\begin{equation}\label{phidY}
\phii^\prime(t_{ij})d_Y(y_i,y_j)\leq 2 \text{~or~} t_{ij}=0.
\end{equation}
Using Assumptions~\ref{ass} and $t_{13}\leq t_{23}$, 
the two cases in \eqref{phidY} for $t_{23}$ give
\[
\alpha\pip(t_{13})d_Y(y_2, y_3)=\alpha\pip(t_{23})d_Y(y_2, y_3) \leq \phii^\prime(t_{23})d_Y(y_2,y_3) \leq 2,
\]
or 
\[\pip(t_{13})d_Y(y_2, y_3)\leq \pip(0)d_Y(y_2, y_3)\leq \Delta/2.\]
Therefore, $\pip(t_{13})d_Y(y_2, y_3)\leq \max\{2\alpha^{-1}, \Delta/2\}.$
Hence, using \eqref{G12}, we have
\begin{align*}
G[z_1,z_2](t_{12}) &\ge  G[z_1,z_3](t_{13})-\max \{2\alpha^{-1}, \Delta/2\}-2\delta\\
&\ge \min\{G[z_1,z_3](t_{13}), G[z_2,z_3](t_{23})\}-(2\alpha^{-1}+\Delta/2+2\delta).
\end{align*}
Whence we conclude that~\eqref{product-compare} holds and the 
proof of the theorem is complete.
\end{proof}

\section{Identifying the visual boundary: proof of Theorem~\ref{thm:Main2}}\label{Sec:ProofMain2}

In this section we study Gromov sequences in $Z=X\times_{\pip} Y$ in terms of their components and prove 
Theorem~\ref{thm:Main2}. 

\begin{lem}\label{lem:Gromov sequences Cauchy}
Let $(z_i)_{i\in\N}$ be a sequence in $Z$, with $z_i=(x_i,y_i)$.
Then $(z_i)_{i\in\N}$ is a Gromov sequence in $Z$ if and only if $(x_i)_{i\in\N}$ is a Gromov sequence in $X$ and 
$(y_i)_{i\in\N}$ is Cauchy in $Y$.
\end{lem}

\begin{proof}
Let $(z_i)_{i\in\N}$ be a sequence in $Z$. For each pair of positive integers $i,j$
we use the functions $G[z_i,z_j]$, as defined in~\eqref{eq:Gzz}; 
$G[z_i,z_j]=2t-\varphi(t)d_{Y}(y_{i},y_{j})$ for $t\in [0, (x_i|x_j)_{x_0}+40\delta]$. We choose $t_{ij}\in [0, (x_i|x_j)_{x_0}+40\delta]$ such that
\[
G[z_i,z_j](t_{ij})\ =\ \max_{0\le t\le (x_i|x_j)_{x_0}+40\delta}\, G[z_i,z_j](t).
\]
We  split into three cases in order to derive a lower bound on $t_{ij}$.

\noindent{\bf Case 1}: $t_{ij}=(x_i|x_j)_{x_0}+40\delta$. Therefore, as $G[z_i,z_j]$ is a differentiable function, 
we have  $G[z_i,z_j]^\prime(t_{ij})\ge 0$, and so using Assumptions \ref{ass}, we obtain
\[
\phii(t_{ij})d_{Y}(y_i,y_j)\le \frac{2}{\alpha}.
\]

In the next two cases we must have $d_Y(y_i,y_j)\ne 0$.

\noindent{\bf Case 2}: $0<t_{ij}< (x_i|x_j)_{x_0}+40\delta$. Then 
$G[z_i,z_j]^\prime(t_{ij})=0$, that is,
$\phii'(t_{ij})=\frac{2}{d_Y(y_i,y_j)}$. Hence, using Assumptions~\ref{ass},
\[
t_{ij}=(\phii')^{-1}(\tfrac{2}{d_Y(y_i,y_j)}) \qquad \mbox{ and } \qquad \phii(t_{ij})d_{Y}(y_i,y_j)\le \frac{2}{\alpha}.
\]

\noindent{\bf Case 3}: $t_{ij}=0$. Then $G[z_i,z_j]^\prime(0)\le 0$, and so
$\phii'(0)d_Y(y_i,y_j)\ge 2$. 

The above three cases, together, establish a lower bound for $t_{ij}$.
We now establish the claim stated in the lemma.

($\Leftarrow$) Suppose first that $(x_i)_{i\in\N}$ is a Gromov sequence in $X$
and $(y_i)_{i\in\N}$ is a Cauchy sequence in $Y$. 
That is, $\liminf_{i,j\to\infty}(x_i|x_j)_{x_0}=\infty$ and $\lim_{i,j\to\infty}d_Y(y_i,y_j)=0$. 
By Lemma~\ref{lem:product}, it suffices to show that $G[z_i,z_j](t_{ij})\to \infty$ as $i,j\to \infty$.
Note that since $d_Y(y_i,y_j)\to 0$ as $i,j\to \infty$, the situation of Case~3 can occur for at most finitely many $i,j$.
From the above three cases, we obtain that when $i$ and $j$ are large enough,
\[
t_{ij}\ge \min\, \bigg\lbrace (x_i|x_j)_{x_0}+40\delta,\ (\phii')^{-1}(\tfrac{2}{d_Y(y_i,y_j)})\bigg\rbrace, 
\]
and
\[
\phii(t_{ij})d_{Y}(y_i,y_j)\le \frac{2}{\alpha}.
\]
Here we adopt the convention that if $d_Y(y_i,y_j)=0$ then $(\pip^\prime)^{-1}(2/d_Y(y_i,y_j)):=\infty$.
Using the estimates above and the fact that $(x_{i})_{i\in\N}$ is Gromov and $(y_{i})_{i\in\N}$ is Cauchy, we obtain
\[
G[z_{i},z_{j}](t_{ij})=2t_{ij}-\varphi(t_{ij})\,d_{Y}(y_{i},y_{j})\geq 2t_{ij}-\frac{2}{\alpha}\to \infty.
\]
This shows that $(z_{i})_{i\in\N}$ is Gromov.

($\Rightarrow$) Now suppose that 
$(z_i)_{i\in\N}$  is a Gromov sequence in $Z$.  
That is $(z_i|z_j)_{z_0}\to \infty$. Then by Lemma~\ref{lem:product}, we have
\[
\max_{0\le t\le (x_i|x_j)_{x_0}+40\delta}\, [2t-\phii(t)d_Y(y_i,y_j)]\to \infty.
\]
Then $t_{ij}\to \infty$ and $\phii(t_{ij})\,d_{Y}(y_i,y_j)\le \frac{2}{\alpha}$.
Since $t_{ij}\leq (x_i|x_j)_{x_0}+40\delta$, we deduce that $(x_i|x_j)_{x_0}\to \infty$ and
\[
d_Y(y_i,y_j)\le \frac{2}{\alpha \phii(t_{ij})}\to 0.
\]
Therefore, $(x_{i})_{i\in\N}$ is Gromov and $(y_{i})_{i\in\N}$ is Cauchy as required.
\end{proof}

\begin{remark}\label{gromovsubsequences}
We recall two facts about Gromov sequences \cite[5.3 Lemma]{VaisalaExpo} that will be useful in the proof of the next 
lemma.
First, a Gromov sequence is equivalent to each of its subsequences. Second, 
if $(x_{i})_{i\in \N}$ and $(y_{i})_{i\in \N}$ are two equivalent Gromov sequences then the sequence 
$(x_{1},y_{1},x_{2},y_{2},\cdots )$ is a Gromov sequence equivalent to both $(x_{i})_{i\in \N}$ and $(y_{i})_{i\in \N}$.
\end{remark}

\begin{lem}\label{lem:well-defined bijection}
The map $\Phi=(\Phi_1,\Phi_2) \colon\partial_G Z\to\partial_G X\times Y$ 
given by 
\[
\Phi([(x_i,y_i)_{i\in\N}])=([(x_i)_{i\in\N}],\lim_{i\to\infty}y_i)
\]
is a well-defined bijection.
\end{lem}

\begin{proof}
We first show that $\Phi$ is well-defined. Suppose that $(x_i,y_i)_{i\in\N}$ and $(\widetilde{x}_i,\widetilde{y}_i)_{i\in\N}$ 
are Gromov sequences in $Z$ such that $(x_i,y_i)_{i\in\N} \sim (\widetilde{x}_{i},\widetilde{y}_{i})_{i\in\N}$,
see Definition~\ref{gromovboundary}. Consider the sequence 
$(a_{i},b_{i})_{i\in\N}$, where $(a_{i})_{i\in\N}=(x_{1},\widetilde{x}_{1},x_{2},\widetilde{x}_{2},\cdots)$ and 
$(b_{i})_{i\in\N}=(y_{1},\widetilde{y}_{1},y_{2},\widetilde{y}_{2},\cdots)$. By Remark~\ref{gromovsubsequences}, the sequence 
$(a_{i},b_{i})_{i\in\N}$ is Gromov in $Z$. Hence by Lemma~\ref{lem:Gromov sequences Cauchy}, 
$(a_{i})_{i\in\N}$ is a Gromov sequence in $X$ and $(b_{i})_{i\in\N}$ is Cauchy in $Y$. 
Using Remark~\ref{gromovsubsequences} again, we have that $(x_{i})_{i\in\N}\sim (\widetilde{x}_{i})_{i\in\N}$. 
Since $(b_{i})_{i\in\N}$ is Cauchy in $Y$ 
and $Y$ is complete (Assumptions~\ref{ass}), $\lim_{i\to \infty}y_{i}=\lim_{i\to \infty}\widetilde{y}_{i}$. Therefore, $\Phi$ is well-defined.

To see that $\Phi$ is surjective, consider $([(x_i)_{i\in\N}], y)\in \partial_G X\times Y$. Let $(y_{i})_{i\in\N}$ be the 
constant sequence with all terms equal to $y$. Lemma~\ref{lem:Gromov sequences Cauchy} then 
implies that $[(x_{i},y_{i})_{i\in\N}]\in \partial_{G}Z$ and clearly $\Phi([(x_{i},y_{i})_{i\in\N})= ([(x_i)_{i\in\N}], y)$.

To see that $\Phi$ is injective, suppose $\Phi([(x_i,y_i)_{i\in\N}])=\Phi([(\widetilde{x}_i,\widetilde{y}_i)_{i\in\N}])$. Then 
$[(x_{i})_{i\in\N}]=[(\widetilde{x}_{i})_{i\in\N}]$ so $(x_{i})_{i\in\N}\sim (\widetilde{x}_{i})_{i\in\N}$. 
Also $\lim_{i\to \infty}y_{i}=\lim_{i\to \infty}\widetilde{y}_{i}$. 
It follows from Remark~\ref{gromovsubsequences} that the sequence 
$(a_{i})_{i\in\N}=(x_{1},\widetilde{x}_{1},x_{2},\widetilde{x}_{2},\cdots)$ is 
Gromov in $X$. Clearly the sequence $(b_{i})_{i\in\N}=(y_{1},\widetilde{y}_{1},y_{2},\widetilde{y}_{2},\cdots)$ is convergent in $Y$. Hence, 
by Lemma~\ref{lem:Gromov sequences Cauchy}, $(a_{i},b_{i})_{i\in\N}$ is a Gromov sequence in $Z$. Again using 
Remark~\ref{gromovsubsequences}, $(x_{i},y_{i})_{i\in\N} \sim (\widetilde{x}_{i},\widetilde{y}_{i})_{i\in\N}$ so that 
$[(x_{i},y_{i})_{i\in\N}]=[(\widetilde{x}_{i},\widetilde{y}_{i})_{i\in\N}]$. 
Thus $\Phi$ is injective.
\end{proof}

Recall from \eqref{eq:phi inverse extension} that we denote the zero extension of $\phii^{-1}$ to $[0,\infty)$ by $\phiinbar$.
We also extend $\phiinbar$ to $\infty$ by setting $\phiinbar(\infty)=\infty$, and so $\phiinbar:[0,\infty]\to[0,\infty]$.

\begin{lem}\label{lem:max bounds}
    Let $D\ge 0$.  Then there is a constant $C_1>0$ so that for every $0\le d\le D$ and $L\ge 0$, we have 
    \begin{equation}\label{eq:two sided minimum estimate}
        -C_1+\min\left\{L,\,\phiinbar\left(\frac{2}{\alpha d}\right)\right\}
        \le\max_{0\le t\le L}[t-\phii(t)\, d/2]
        \le C_1+\min\left\{L,\,\phiinbar\left(\frac{2}{\alpha d}\right)\right\}.
    \end{equation}
    Here the constant $C_1$ can be taken to be
    \[
    C_1=\max\left\{\frac{1}{\alpha},\,D\frac{\phii(0)}{2}+\phiinbar\left(\frac{\phii'(0)}{\alpha}\right)\right\}.
    \]
\end{lem}

\begin{proof}
The desired inequality clearly holds when $d=0$ or $L=0$. Hence we suppose that $d>0$ and $L>0$. 
Consider the function $G(t):=t-\phii(t)\,d/2$ and set $M:=\max_{0\le t\le L}G(t)$. We can choose $t_0\in[0,L]$ 
such that $M=t_0-\phii(t_0)\, d/2$.  
We consider the following three cases:

    \noindent{\bf Case 1:} Suppose that $t_0=L$.  In this case, it follows that
    \[
    0\le G'(L)=1-\phii'(L)\,d/2.
    \]
    Hence by Assumptions~\ref{ass}, we have that
    \[
    \phii(0)\le\phii(L)\le\phii'(L)/\alpha\le\frac{2}{\alpha d}.
    \]
    Hence, $L\le \phii^{-1}(2/(\alpha d))=\phiinbar(2/(\alpha d))$, and so we have that 
    \[
    M=L-\phii(L)\, d/2\le L\le \phiinbar\left(\frac{2}{\alpha d}\right).
    \]
    This gives us the right-hand inequality of \eqref{eq:two sided minimum estimate} in this case.  
    For the left-hand inequality, we note that $d\le 2/(\alpha\phii(L))$, and so
    \[
    M=L-\phii(L)\, d/2\ge L-\frac{1}{\alpha}.
    \]

    \noindent{\bf Case 2:} Suppose that $0<t_0<L$.  In this case, we have that $0=G'(t_0)=1-\phii'(t_0)\, d/2$, and so it follows that 
    \[
    d=\frac{2}{\phii'(t_0)}\le\frac{2}{\alpha\phii(t_0)}<\frac{2}{\alpha\phii(0)}.
    \]
    Hence, we have that $t_0\le \phii^{-1}(2/(\alpha d))=\phiinbar(2/(\alpha d))$, and so 
    \[
    M=t_0-\phii(t_0)\, d/2\le t_0\le\min\left\{L,\phiinbar\left(\frac{2}{\alpha d}\right)\right\}. 
    \]
    This gives us the right-hand inequality of \eqref{eq:two sided minimum estimate} in this case.

    To establish the left-hand inequality, first suppose that $0<d\le2/(\alpha\phii(L))$.  In this case, we have that 
    \[
    M=G(t_0)\ge G(L)=L-\phii(L)\, d/2\ge L-\frac{1}{\alpha}.
    \]
    If $d>2/(\alpha\phii(L))$, then $0\le\phii^{-1}(2/(\alpha d))<L$, and so it follows that 
    \[
    M\ge G(\phii^{-1}(2/(\alpha d)))=\phii^{-1}\left(\frac{2}{\alpha d}\right)-\frac{1}{\alpha}=\phiinbar\left(\frac{2}{\alpha d}\right)-\frac{1}{\alpha}.
    \]
    Thus, in either case, the left-hand inequality of \eqref{eq:two sided minimum estimate} holds for Case 2.

   \noindent{\bf Case 3:} Suppose that $t_0=0$. Since $\phiinbar(2/(\alpha d))\ge 0$, we have that 
    \[
    M=G(0)=-\phii(0)\, d/2\le 0\le\min\left\{L,\phiinbar\left(\frac{2}{\alpha d}\right)\right\},
    \]
    and so the right-hand inequality of \eqref{eq:two sided minimum estimate} holds.  

    To prove the left-hand inequality, we note that 
    \[
    0\ge G'(0)=1-\phii'(0)\, d/2,
    \]
    and so we have that 
    \[
    \frac{2}{\phii'(0)}\le d\le D.
    \]
    Thus, it follows that 
    \begin{align*}
    M=-\phii(0)\, \frac{d}{2}&\ge -D\frac{\phii(0)}{2}
    -\phiinbar\left(\frac{2}{\alpha d}\right)+\phiinbar\left(\frac{2}{\alpha d}\right)\\
    &\ge  -D\frac{\phii(0)}{2}-\phiinbar\left(\frac{\phii'(0)}{\alpha}\right)+\phiinbar\left(\frac{2}{\alpha d}\right).\qedhere
    \end{align*}
    \end{proof}

\begin{lem}\label{lem:boundary product comparison}
    Let $z,\wtil z\in\partial_G Z$, and let 
    $x=\Phi_1(z)$, $\wtil x=\Phi_1(\wtil z)$, $y=\Phi_2(z)$, and $\wtil y=\Phi_2(\wtil z)$.  Then there exists $C_2>0$ such that 
    \begin{align}\label{eq:two sided boundary gromov product}
    -C_2+\min\bigg\{(x|\wtil x)_{x_0}+&40\delta,\,\phiinbar\left(\frac{2}{\alpha d_Y(y,\wtil y)}\right)\bigg\}\nonumber\\
    &\le (z|\wtil z)_{z_0}\le C_2+\min\left\{(x|\wtil x)_{x_0}+40\delta,\,\phiinbar\left(\frac{2}{\alpha d_Y(y,\wtil y)}\right)\right\}.
    \end{align}
Here the constant $C_2$ can be taken to be 
\[
C_2=\frac{1}{\alpha}+43\delta+3\diam(Y)\frac{\phii(0)}{2}+\phiinbar\left(\frac{\phii'(0)}{\alpha}\right).
\]
\end{lem}

\begin{proof}
    Let $(z_i)_{i\in \N}\in z$ and $(\wtil z_i)_{i\in \N}\in\wtil z$, and denote $z_i=(x_i,y_i)$ and $\wtil z_i=(\wtil x_i,\wtil y_i)$.  
    Hence $(x_i)_{i\in\N}\in \Phi_1(z)$ and $(y_i)_{i\in\N}\in\Phi_2(z)$, and similarly $(\widetilde{x}_i)_{i\in\N}\in \Phi_1(\widetilde{z})$ 
    and $(\wtil y_i)_{i\in\N}\in\Phi_2(\wtil z)$.
    By Lemma~\ref{lem:product}, it follows that 
    \[
    (z_i|\wtil z_j)_{z_0}\ge-43\delta+\max_{[0,(x_i|\wtil x_j)_{x_0}+40\delta]}(t-\phii(t)d_Y(y_i,\wtil y_j)/2).
    \]
    Applying Lemma~\ref{lem:max bounds} with $d=d_Y(y_i,\wtil y_j)$, $D=\diam(Y)$, and $L=(x_i|\wtil x_j)_{x_0}+40\delta$, we obtain
    \[
(z_i|\wtil z_j)_{z_0}\ge-43\delta-C_1+\min\left\{(x_i|\wtil x_j)_{x_0}+40\delta,\ \phiinbar\left(\frac{2}{\alpha d_Y(y_i,\wtil y_j)}\right)\right\},
    \]
where $C_1$ is the constant from Lemma~\ref{lem:max bounds} with the choice of $D=\diam(Y)$.
By Lemma~\ref{lem:Gromov sequences Cauchy}, we know that $(x_i)_{i\in \N}\in x$, $(\wtil x_i)_{i\in \N}\in\wtil x$, 
 and both $(y_i)_{i\in \N}$ and $(\wtil y_i)_{i\in \N}$ are Cauchy in $Y$.  By continuity of $\phiinbar$, it then follows that 
\begin{align*}
    \liminf_{i,j\to\infty}(z_i|\wtil z_j)_{z_0}&\ge-43\delta-C_1
      +\min\left\{\liminf_{i,j\to\infty}(x_i|\wtil x_j)_{x_0}+40\delta,\ \phiinbar\left(\frac{2}{\alpha d_Y(y,\wtil y)}\right)\right\}\\
    &\ge-43\delta-C_1+\min\left\{(x|\wtil x)_{x_0}+40\delta,\ \phiinbar
\left(\frac{2}{\alpha d_Y(y,\wtil y)}\right)\right\}.
\end{align*}
Since $(z_i)_{i\in \N}\in z$ and $(\wtil z_i)_{i\in \N}\in\wtil z$ are arbitrary, and 
recalling~\eqref{eq:boundary gromov product}, this gives us the left-hand inequality of \eqref{eq:two sided boundary gromov product}.

To prove the right-hand inequality, let $(x_i)_{i\in \N}\in x$ and $(\wtil x_i)_{i\in \N}\in\wtil x$.  Then, 
setting $z_i=(x_i,y)$ and $\wtil z_i=(\wtil x_i,\wtil y)$, it follows from Lemma~\ref{lem:well-defined bijection} 
that $(z_i)_{i\in \N}\in z$ and $(\wtil z_i)_{i\in \N}\in \wtil z$.  By Lemma~\ref{lem:product} and 
Lemma~\ref{lem:max bounds} with the same choices of $D$, $d$, and $L$ specified in the prior paragraph, we have that 
\begin{align*}
    (z|\wtil z)_{z_0}&\le\liminf_{i,j\to\infty}(z_i|z_j)_{z_0}\\
    &\le\phii(0)\diam(Y)+C_1+\min\left\{\liminf_{i,j\to\infty}(x_i|\wtil x_j)_{x_0}+40\delta,\ \phiinbar\left(\frac{2}{\alpha d_Y(y,\wtil y)}\right)\right\}.
\end{align*}
Since $(x_i)_{i\in \N}\in x$ and $(\wtil x_i)_{i\in \N}\in\wtil x$ are arbitrary, this gives us the right-hand 
side of~\eqref{eq:two sided boundary gromov product}.
\end{proof}

In what follows, for each $0<\eps<\log_e(2)/\delta$ the metric $d_{\eps,X}$ is the visual metric on $\partial_GX$ 
given by the chaining construction 
described in~\eqref{eq:visual metric}.

\begin{lem}\label{lem:homeomorphism}
    Let $\Phi=(\Phi_1,\Phi_2)$ be the bijection given by Lemma~\ref{lem:well-defined bijection}, and let 
    $0<\eps\le\log_e(2)/\delta_\phii$, where $\delta_\phii$  
    as 
    given by Theorem~\ref{thm:Main1}. 
    Then there exists  a visual metric $d_{\eps, Z}$ on $\partial_GZ$ and a constant $C\ge 1$ such that for all 
    $z,\wtil z\in\partial_GZ$, we have 
    \begin{equation}\label{eq:bijection metric comparison}
        C^{-1}d_{\eps,Z}(z,\wtil z)\le d_{\eps,X}\left(x,\wtil x)\right)
           +\exp\left\{-\eps\phiinbar\left(\frac{2}{\alpha d_Y\left(y,\wtil y\right)}\right)\right\}\le Cd_{\eps,Z}(z,\wtil z),
    \end{equation}
    where $x=\Phi_1(z)$, $\wtil x=\Phi_1(\wtil z)$, $y=\Phi_2(z)$, and $\wtil y=\Phi_2(\wtil z)$.  
    Here, the constant $C\ge 1$ can be taken to be
    \[
    C= 2^{3+\tfrac{C_2+40\delta}{\delta_\pip}}, 
    \]
    where $C_2$ is the constant from Lemma~\ref{lem:boundary product comparison}.
\end{lem}

\begin{proof}
We set $\eps_0=\log_e(2)/\delta_\phii$.
    By our choice of $\eps$, and since $Z$ is $\delta_\phii$-hyperbolic with respect to the special base point $(x_0,y_0)\in Z$, 
    we can choose the visual metric $d_{\eps,Z}$ as described in~\eqref{eq:visual metric} with $z_0=(x_0,y_0)$ playing the role of $w_0$. 
    Note that $\delta_\pip> \delta$, and so this range of $\eps$ also satisfies the requirement that $\eps<\log_e(2)/\delta$.
    Again,
    we have from \eqref{eq:visual metric comparison} and Lemma~\ref{lem:boundary product comparison} that 
    \begin{align*}
        d_{\eps,Z}(z,\wtil z)\le\tilde d_{\eps,Z}(z,\wtil z)&=e^{-\eps(z|\wtil z)_{z_0}}\\
        &\le e^{\eps C_2}\exp\left\{-\eps\min\left\{(x|\wtil x)_{x_0}+40\delta,\ \phiinbar\left(\frac{2}{\alpha d_Y(y,\wtil y)}\right)\right\}\right\}\\
        &=e^{\eps C_2}\max\left\{e^{-\eps[(x|\wtil x)_{x_0}+40\delta]},\,\exp\left\{-\eps\phiinbar\left(\frac{2}{\alpha d_Y(y,\wtil y)}\right)\right\}\right\}\\
        &\le 2e^{\eps_0C_2}\left(d_{\eps,X}(x,\wtil x)+\exp\left\{-\eps\phiinbar\left(\frac{2}{\alpha d_Y(y,\wtil y)}\right)\right\}\right).
    \end{align*}
    Here we have also used \eqref{eq:visual metric comparison} in the last inequality.   
    This gives us the left-hand inequality of \eqref{eq:bijection metric comparison}.

Note that as $\eps\le \log_e(2)/\delta_\pip$, we have $e^{\eps \delta_\pip}\le 2$.
    To obtain the right-hand inequality, we similarly apply~\eqref{eq:visual metric comparison} and 
    Lemma~\ref{lem:boundary product comparison} as follows:
    \begin{align*}
        d_{\eps,Z}(z,\wtil z)&\ge 4^{-1}e^{-\eps(z|\wtil z)_{z_0}}\\
        &\ge 4^{-1}e^{-\eps_0C_2}\exp\left\{-\eps\min\left\{(x|\wtil x)_{x_0}+40\delta, \ \phiinbar\left(\frac{2}{\alpha d_Y(y,\wtil y)}\right)\right\}\right\}\\
        &\ge 8^{-1}e^{-\eps_0(C_2+40\delta)}\left(d_{\eps,X}(x,\wtil x)+\exp\left\{-\eps\phiinbar\left(\frac{2}{\alpha d_Y(y,\wtil y)}\right)\right\}\right).\qedhere
    \end{align*}
\end{proof}

We are now ready to prove Theorem~\ref{thm:Main2}.

\begin{proof}[Proof of Theorem~\ref{thm:Main2}]
The map $\Phi$ is a well-defined bijection from $\partial_{G}Z$ to $\partial_{G}X \times Y$ by 
Lemma~\ref{lem:well-defined bijection}. The required comparison estimates then hold by 
Lemma~\ref{lem:homeomorphism}, and it also follows that $\Phi$ is a homeomorphism between $\partial_GZ$ equipped 
with the visual metric $d_{\eps,Z}$, and the product topology on the Cartesian product of the two metric spaces $(\partial_GX, d_{\eps,X})$
and $(Y,d_Y)$.
\end{proof}

\section{Examples}

We conclude this paper by examining the 
relationship between the Gromov boundaries of the warped product and component spaces
for particular choices of warping function $\phii$. This serves to illustrate Theorem~\ref{thm:Main2}.
We also consider quasisymmetric non-equivalence between the boundaries corresponding to different
choices of $\pip$. Two metrics $d_1$ and $d_2$ on a set $W$ are quasisymmetric equivalent to each other
if there is a homeomorphism $\eta:[0,\infty)\to[0,\infty)$ such that for every triple 
$w_1,w_2,w_3\in W$ with $w_1\ne w_2$ we have
\[
\frac{d_2(w_1,w_3)}{d_2(w_1,w_2)}\le \eta\left(\frac{d_1(w_1,w_3)}{d_1(w_1,w_2)}\right).
\]
It is known that if $W$ is uniformly perfect with respect to the two topologies generated by $d_1$ and $d_2$, then
the function $\eta$ can be chosen to be a power function, see for instance~\cite{BHK}.

\begin{example}\label{example:exponential}
    Fix spaces $X,Y$ as in Assumptions \ref{ass}. Let $\alpha>0$ and consider $\phii(t)=e^{\alpha t}$ 
    with corresponding warped product space $Z$.  With this choice, we have that 
    \[
    \phiinbar(t)=\begin{cases}
        0,&0\le t<1\\
        \frac{1}{\alpha}\log_e(t),&t\ge 1.
    \end{cases}
    \]
    Thus, if $z,\wtil z\in\partial_GZ$, $y=\Phi_2(z)$, $\wtil y=\Phi_2(\wtil z)$, 
    then 
    \begin{align*}
       \exp\left\{-\eps\phiinbar\left(\frac{2}{\alpha d_Y\left(y,\wtil y)\right)}\right)\right\}&=\begin{cases}
           1,&\frac{2}{\alpha}\le d_Y(y,\wtil y)\le\diam(Y)\\
           \left(\frac{\alpha}{2}d_Y(y,\wtil y)\right)^{\eps/\alpha},& 0\le d_Y(y,\wtil y)\le\frac{2}{\alpha}.
       \end{cases}\\
       &\simeq d_Y(y,\wtil y)^{\eps/\alpha},
    \end{align*}
    with comparison constants depending on $\alpha$, $\diam(Y)$,
    and $\eps$.
   Recall that as $\phii(0)=1$, the Gromov hyperbolicity constant $\delta_\phii$ of $Z$
   is given by 
   \[
   \delta_\pip=\frac{1}{\alpha}+44\delta+\frac32\diam(Y)>\delta.
   \]
   However, for each $0<\eps<\log_e(2)/\delta$ such that $\eps\le \alpha$, the function 
   \[
   d_Z(z,\widetilde{z}):=d_{\eps,X}(\Phi_1(z),\Phi_1(\wtil z))+d_Y(\Phi_2(z),\Phi_2(\wtil z))^{\eps/\alpha}
   \]
   is a metric on $\partial_GZ$. A direct computation with Lemma~\ref{lem:boundary product comparison} shows us that
   $d_Z$ is a visual metric, as in Definition~\ref{def:VisualBdy}, on $\partial_GZ$ with parameter $\eps$. Hence the range of
   $\eps$ identified in Theorem~\ref{thm:Main2} is not always optimal.
   
   Note that if $\partial_GX$ has at least one cluster point and $Y$ has at least two points, then the two warped spaces $X\times_{\pip_\alpha} Y$
   and $X\times_{\pip_\beta} Y$ are not quasisymmetrically equivalent, where $\pip_\alpha$ and $\pip_\beta$ are given by
   $\pip_\alpha(t)=e^{\alpha t}$ and $\pip_\beta(t)=e^{\beta t}$ with $0<\alpha<\beta$.
   
    If $X=[0,\infty)$ is equipped with the Euclidean metric, 
    then $X$ is $0$-hyperbolic and $\partial_GX$ is 
    a singleton.  In this case, we see that 
    \begin{equation*} 
    d_{Z}(z,\wtil z)= d_Y(\Phi_2(z),\Phi_2(\wtil z))^{\eps/\alpha}
    \end{equation*}
    is a visual metric on $\partial_GZ$.
    Hence for every positive $\eps\le \alpha$,
    the bijection $\Phi:\partial_GZ\to Y$ is a quasisymmetric homeomorphism.  
    In this manner, 
Theorem~\ref{thm:Main2} generalizes \cite[Lemma~7.22, Theorem 1.7]{KKSS}, proven when $X=[0,\infty)$, to the 
case of more general Gromov hyperbolic spaces $X$.    Moreover, with this choice of $X$ we have that 
$X\times_{\pip_\alpha} Y$ and $X\times_{\pip_\beta} Y$ are quasisymmetrically equivalent.
\end{example}

\begin{example}\label{example:superexponential}
Fix spaces $X,Y$ as in Assumptions \ref{ass}.
    Consider $\phii(t)=e^{e^{\alpha t}}$ for some $\alpha>0$, with corresponding warped product space $Z$. Then 
      \[
    \phiinbar(t)=\begin{cases}
        0,&0\le t<e\\
        \frac{1}{\alpha}\log_e\log_e(t),&t\ge e.
    \end{cases}
    \]
   Thus, for $z,\wtil z\in\partial_GZ$, $y=\Phi_2(z)$, and $\wtil y=\Phi_2(\wtil z)$, we have   
   \begin{align*}
       \exp\left\{-\eps\phiinbar\left(\frac{2}{\alpha d_Y\left(y,\wtil y\right)}\right)\right\}=\begin{cases}
           1,&\frac{2}{\alpha e}\le d_Y(y,\wtil y)\le\diam(Y)\\
           \left[-\log_e\left(\frac{\alpha d_Y(y,\wtil y)}{2}\right)\right]^{-\frac{\eps}{\alpha}},& 0\le d_Y(y,\wtil y)\le\frac{2}{\alpha e}.
       \end{cases}
    \end{align*}
    Hence, for $0<\eps\le\log_e(2)/\delta_\pip$, by Theorem~\ref{thm:Main2},
    \begin{align*}
        d_{\eps, Z}(z,\wtil z)\simeq d_{\eps, X}(\Phi_1(z),\Phi_1(\wtil z))+\begin{cases}
           1,&\frac{2}{\alpha e}\le d_Y(y,\wtil y)\le\diam(Y)\\
           \left[-\log_e\left(\frac{\alpha d_Y(y,\wtil y)}{2}\right)\right]^{-\frac{\eps}{\alpha}},& 0\le d_Y(y,\wtil y)\le\frac{2}{\alpha e}.
       \end{cases}
    \end{align*}
    Here the comparison constant depends only on $\alpha$, $\delta$, and $\diam(Y)$, as $\phii(0)=e$ and 
    $\phiinbar(\phii'(0)/\alpha)=0$ in this case.

    If $X=[0,\infty)$, then for this choice of $\phii$, the visual metric $d_{\eps, Z}$ is not bi-H\"older equivalent to $d_Y$.
    Indeed, for any $\beta>0$ and $C\ge 1$, the inequality
    \[
    \left[-\log_e\left(\frac{\alpha d_Y(y,\wtil y)}{2}\right)\right]^{-\frac{\eps}{\alpha}}\le Cd_Y(y,\wtil y)^\beta
    \]
    fails for $d_Y(y,\wtil y)$ sufficiently close to zero.
    
    This example exhibits a quasisymmetric phenomenon the previous example has; if $\partial_GX$ has at least one cluster point
    and $Y$ has at least two points, then change in the parameter $\alpha$ determining the warping function $\pip(t)=e^{e^{\alpha t}}$
    results in a different quasisymmetric class of $\partial_G (X\times_\pip Y)$, while we have the same quasisymmetry class if $X=[0,\infty)$. 
\end{example}

The next example indicates that the warped product yields a richer geometry at large scale, with influence from $Y$, than the regular Cartesian product.

\begin{example}\label{example:Cartesian}
Observe that the Gromov boundary in both of the above examples are identified with $\partial_GX\times Y$, as
demonstrated by Lemma~\ref{lem:well-defined bijection}.
In the above two examples, the metric on this Gromov boundary belongs to different quasisymmetry classes of
$\partial_GX\times Y$ (see the discussion in~\cite{BHK}). There is another choice of $\pip$ that has not been considered in
this paper--namely, the choice of $\pip\equiv 1$. In this case, the ``warped'' product space is merely the Cartesian product
$X\times Y$. The large scale geometry of this space is determined solely by that of $X$, with $Y$ playing no role; 
$\partial_G(X\times Y)=\partial_GX$. The assumptions stipulated in Assumptions~\ref{ass} guarantee that $Y$ plays a significant
role in the large scale geometry of the warped product, and the quasisymmetry class of metrics on the associated Gromov
boundary is determined also by the behavior of $\pip$, see Theorem~\ref{thm:Main2}.
\end{example}


\begin{thebibliography}{A}
\frenchspacing
%
\bibitem{AB} S.~Alexander, R.~L.~Bishop:
\emph{A cone splitting theorem for Alexandrov spaces.}
Pacific J. Math. {\bf 218} (2005), no. 1, 1--15.
\bibitem{AB1} S.~Alexander, R.~L.~Bishop:
\emph{Curvature bounds for warped products of metric spaces.}
Geom. Funct. Anal. {\bf 14} (2004), no. 6, 1143--1181.
\bibitem{AB2} S.~Alexander, R.~L.~Bishop:
\emph{Warped products admitting a curvature bound.}
Adv. Math. {\bf 303} (2016), 88--122. 
\bibitem{BO} R.~L.~Bishop, B.~O'Neill:
\emph{Manifolds of negative curvature.}
Trans. Amer. Math. Soc. {\bf 145} (1969), 1--49
\bibitem{Brendle} S.~Brendle:
\emph{Constant mean curvature surfaces in warped product manifolds.}
Publ. Math. Inst. Hautes Études Sci. {\bf 117} (2013), 247--269.
\bibitem{BridsonHaefliger} M. Bridson, A. Haeflger:
\emph{Metric spaces of non-positive curvature.}
Grundlehren der mathematischen Wissenschaften [Fundamental Principles of Mathematical Sciences]
{\bf 319}. Springer-Verlag, Berlin, 1999. xxii+643 pp.
\bibitem{BGM}  M.~Brozos-V{\'a}zquez, E.~García-R{\'i}o, D.~Moj{\'o}n-{\'A}lvarez:
\emph{Conformally weighted Einstein manifolds: the uniqueness problem.}
J. Differential Equations {\bf 450} (2026), Paper No. 113711, 21 pp.
 \bibitem{BHK} M. Bonk, J. Heinonen, P. Koskela:
 \emph{Uniformizing Gromov hyperbolic spaces.}
 Ast\'erisque {\bf 270} (2001), viii+99 pp. 
\bibitem{BuSch} S. Buyalo, V. Schroeder:
\emph{Elements of asymptotic geometry.}
EMS Monographs in Mathematics. European Mathematical Society (EMS), Z\"urich, 2007. 
\bibitem{Chen} C.-H.~Chen:
\emph{Warped products of metric spaces of curvature bounded from above.}
Trans. Amer. Math. Soc. {\bf 51} (1999) No.~12,  4727--4740.
\bibitem{CDP} M.~Coornaert, T.~Delzant, A.~Papadopoulos:
\emph{G\'eom\'etrie et th\'eorie des groupes.}
Lecture Notes in Math. {\bf 1441}, Springer, Berlin, 1990.
\bibitem{D} H.-T.~Dung:
\emph{On the structure at infinity of complete smooth metric measure spaces with a weighted Poincar\'e inequality.}
Manuscripta Math. {\bf 176} (2025), no. 4, Paper No. 45, 30 pp.
\bibitem{GGN} S. Ghinassi, V. Giri, E. Negrini:
\emph{A step towards the tensorization of Sobolev spaces.}
Ann. Fenn. Math. {\bf 50} (2025), no. 2, 721--740.
\bibitem{GigMarc} N.~Gigli, F.~Marconi:
\emph{A general splitting principle on RCD spaces and applications to spaces with positive spectrum.}
J. Geom. Anal. {\bf 35} (2025), no. 11, Paper No. 330, 50 pp.
\bibitem{HK} P. Haj\l asz, P. Koskela:
\emph{Sobolev met Poincar\'e.}
Memoirs AMS {\bf 145} No.~688 (2000) ix+101. 
\bibitem{Hei} J. Heinonen:
\emph{Lecture notes on analysis in metric spaces.}
Springer Universitext, Springer Verlag New York (2001).
\bibitem{HKSTbook} J. Heinonen, P. Koskela, N. Shanmugalingam, J. T. Tyson:
\emph{Sobolev Spaces on Metric Measure Spaces: An Approach Based on Upper Gradients.}
Cambridge University Press, 2015.
\bibitem{HSX} D.~Herron, N. Shanmugalingam, X.~Xie:
\emph{Uniformity from Gromov hyperbolicity.}
Illinois J. math. {\bf 52} No.~4 (2008), 1065--1109.
\bibitem{KKSS} I. Kangasniemi, J. Kline, N. Shanmugalingam, G. Speight:
\emph{Warped products, solid hyperbolic fillings, and the identity $ D^{1, p}= N^{1, p}+\mathbb {R} $}. (2025) arXiv preprint arXiv:2508.01857.
\bibitem{Ogawa} Y.~Ogawa:
\emph{On conformally flat spaces with warped product Riemannian metric.}
Natur. Sci. Rep. Ochanomizu Univ. {\bf 29} (1978), no. 2, 117--127.
{\tt https://cir.nii.ac.jp/crid/1050282677924927104}
\bibitem{Tsukada} K. Tsukada:
\emph{Eigenvalues of the Laplacian of warped product.}
Tokyo J. Math. {\bf 3} (1980), no. 1, 131--136.
\bibitem{Vaisala1971}  J. V\"ais\"al\"a:
\emph{Lectures on $n$-Dimensional Quasiconformal Mappings.}
Springer Lecture Notes in Mathematics {\bf 229} 1971, {\tt https://doi.org/10.1007/BFb0061216}, pp. XVI, 152.
\bibitem{VaisalaExpo} J. V\"ais\"al\"a:
\emph{Gromov hyperbolic spaces.}
Expo. Math. {\bf 23} (2005), 187--231.
\end{thebibliography}
\end{document}